\documentclass[a4paper,11pt]{amsart} %
\usepackage{amsmath,amsfonts,amssymb,graphicx,url,amsthm,hyperref}
\usepackage[capitalize,nameinlink]{cleveref}
\usepackage[left=3cm,right=3cm,top=3cm,bottom=3cm]{geometry}
\usepackage[nocompress]{cite}
\usepackage[T1]{fontenc}

\usepackage{caption}
\usepackage[labelformat=simple]{subcaption}
\usepackage{graphicx,color,nicefrac} 
\usepackage{amssymb, amsmath, amsbsy,mathtools}
\usepackage{amsmath,amsfonts}
\usepackage[export]{adjustbox}
\usepackage{booktabs}
\usepackage{leftindex}
\usepackage{tikz}
\usepackage{pgf}
\usepackage{import}
\usepackage{bbm}

\usepackage[normalem]{ulem}

\usepackage{algorithm}
\usepackage{algpseudocode}
\usepackage{bm} %
\usepackage{dsfont}

\usepackage{tikz, pgfplots}
\pgfplotsset{compat=newest}

\usepackage{algorithm}
\usepackage{algpseudocode}
\usepackage{xcolor}
\usepackage{tikz,pgfplots,pgfplotstable}
\pgfplotsset{compat=newest}
\usetikzlibrary{external}
\usepackage[mathlines]{lineno} 
\usepackage{multirow}
\usepackage{booktabs}
\newtheorem{theorem}{Theorem}[section]

\newtheorem{remark}[theorem]{Remark}
\newtheorem{assumption}[theorem]{Assumption}

\def\letters{a,b,c,d,e,f,g,h,i,j,k,l,m,n,o,p,q,r,s,t,u,v,w,x,y,z}
\def\Letters{A,B,C,D,E,F,G,H,I,J,K,L,M,N,O,P,Q,R,S,T,U,V,W,X,Y,Z}
\makeatletter
\@for \@l:=\Letters \do{%
  \expandafter\edef\csname\@l bb\endcsname{%
  \noexpand\ensuremath{\noexpand\mathbb{\@l}}}%
  \expandafter\edef\csname\@l bf\endcsname{{\noexpand\bf \@l}}%
  \expandafter\edef\csname\@l cal\endcsname{%
  \noexpand\ensuremath{\noexpand\mathcal{\@l}}}%
  \expandafter\edef\csname\@l eu\endcsname{%
  \noexpand\ensuremath{\noexpand\EuScript{\@l}}}%
  \expandafter\edef\csname\@l frak\endcsname{%
  \noexpand\ensuremath{\noexpand\mathfrak{\@l}}}%
  \expandafter\edef\csname\@l rm\endcsname{{\noexpand\rm \@l}}%
  \expandafter\edef\csname\@l scr\endcsname{%
  \noexpand\ensuremath{\noexpand\mathscr{\@l}}}%
}
\@for \@l:=\letters \do{%
  \expandafter\edef\csname\@l bf\endcsname{{\noexpand\bf \@l}}%
  \expandafter\edef\csname\@l frak\endcsname{%
  \noexpand\ensuremath{\noexpand\mathfrak{\@l}}}%
  \expandafter\edef\csname\@l scr\endcsname{%
  \noexpand\ensuremath{\noexpand\mathscr{\@l}}}%
}
\makeatother
\definecolor{shadecolor}{rgb}{0.6, 0.6, 0.6} 
\definecolor{darkgreen}{rgb}{0, 0.6, 0}
\newcommand{\isdef}{\mathrel{\mathrel{\mathop:}=}}

\newcommand{\R}{\mathbb{R}}
\newcommand{\N}{\mathbb{N}}

\newcommand{\bs}{\boldsymbol}

\usepackage{todonotes}
\usepackage{pgfplots}
\pgfplotsset{compat=1.18}

\definecolor{gold}{RGB}{255,215,0}
\definecolor{magenta_mat}{RGB}{255,0,255}

\begin{document}
 \title[Smoothness analysis via sharp direct and inverse statements]{Sobolev Algorithm for Local Smoothness Analysis (SALSA) via sharp direct and
 inverse statements} %
\author{S.\,Avesani, L.\,Ling, F.\,Marchetti, T.\,Wenzel}

\address{
Sara Avesani,
IDSIA,
USI Lugano,
Via la Santa 1, 6962 Lugano, Svizzera.}
\email{sara.avesani@usi.ch}

\address{
Leevan Ling,
Department of Mathematics, 
Hong Kong Baptist University, 
Kowloon Tong, 
Hong Kong}
\email{lling@hkbu.edu.hk}

\address{
Francesco Marchetti,
University of Padua, 
Via Trieste, 63, 35131 Padova, Italia}
\email{francesco.marchetti@unipd.it}

\address{
Tizian Wenzel,
LMU Munich \& MCML Munich,
Theresienstraße 39, 80333 München, Germany.}
\email[Corresponding author]{wenzel@math.lmu.de}

\begin{abstract}
We extend sharp direct and inverse approximation statements for kernel-based methods for finitely smooth kernels, 
i.e.\ those whose native spaces are norm-equivalent to Sobolev spaces. 
In particular, our inverse results are now formulated for a broad class of approximation schemes beyond interpolation, 
extending existing theory. 

Building on these results, we propose a novel \textit{Sobolev Algorithm for Local Smoothness Analysis (SALSA)} for detecting local smoothness properties of target data, 
including their degree of smoothness and non-smoothness. 
The method is rigorously grounded based on the sharp direct and inverse statements.
Numerical experiments in various settings highlight the effectiveness of the proposed algorithm. \\



\end{abstract}

\maketitle

\section{Introduction}

In many applications, one frequently deals with signals that are of irregular nature. 
This can be due to structural properties, as in the popular case of images that naturally exhibit edges and jumps in pixel values. 
Alternatively,
they may arise from noise or missing data, as happens in time-series analysis. 
Consequently, determining the smoothness class of the function that generates the data is crucial for developing effective data-analysis and reconstruction techniques.

In this direction, most of the research focuses on the \textit{edge detection} problem, 
which consists in finding jump discontinuities in images and computer vision applications \cite{Chikmurge18,Shrivakshan12,WilliamThomas15}; 
see the review \cite{Jing22} and references therein.
Indeed, it is well known that the presence of discontinuities gives rise to the so-called \textit{Gibbs phenomenon} in various approximation approaches, 
which causes distortions and oscillations that are not limited to regions near leap points, 
but are likely to extend to significant areas of the domain \cite{Fornberg10,Jerri98,Zygmund03}. 
Well-established techniques in the literature, such as Canny \cite{Canny1986Edge} or Sobel \cite{kittler1983accuracy} detectors, and spatial-frequency methods \cite{MallatZhong1992MultiscaleEdges}, 
usually require the presence of an underlying grid. 
However, in many situations, we need to deal with scattered data, and thus efforts have been made to develop methods that can also treat this setting \cite{BozziniRossini2013DiscontinuityCurves,CramptonMason2005FaultLines,LenarduzziSchaback2017KernelAdaptive}.
In particular, many of such methods exploit kernel-based models, thanks to their flexibility in reconstructing functions at various nodal designs \cite{jung2009iterative,jung2011iterative}. 
In \cite{romani2019edge}, by a Fourier analysis of the cardinal coefficients, the authors showed that the Radial Basis Function (RBF) interpolant coefficients corresponding to the nodes next to the edge points take larger absolute values than those lying in smooth regions, 
and they made use of this observation to derive an RBF-based edge detector. 

Later in \cite{DeMarchiEtAl2020ShapeDriven}, a different edge detection algorithm was constructed by properly thresholding these coefficients and then extrapolating the position of jumps in the domain via Support Vector Machines (SVMs). Other approaches that directly use machine or deep learning methods have also been considered in less and more recent years \cite{BertasiusShiTorresani2015DeepEdge,EtemadChellappa1993NeuralEdge,FloresVidalEtAl2019GlobalEdgeDetection,HeEtAl2019BDCN,LiPaluriRehgDollar2016UnsupervisedEdges}.

Recently, a kernel-free linear-time method for detecting pointwise regularity in non-uniformly sampled multivariate data was introduced in \cite{avesani2025multiresolution}, where the authors exploit a multiresolution analysis tailored to scattered data and analyze the local decay of samplet coefficients, see \cite{harbrecht2022samplets}.

In general, the presence of irregularities of any order in the target function and the convergence properties of reconstruction methods are strongly intertwined. 
For this, \textit{direct statements} quantify the rate of approximation based on smoothness properties of the target function.
Such statements are well studied in the literature, ranging from meshfree approximation \cite{fasshauer2015kernel,wendland2004scattereddata} to the statistical kernel learning context \cite{AlbaniEtAl2016OptimalConvergence,BissantzHohageMunkRuymgaart2007,GuastavinoBenvenuto2020SpectralConvergence}.
Particularly when interpolating with Sobolev kernels, direct statements provide error estimates in various norms, where the achievable rate usually scales with the smoothness of the function to be approximated, and is frequently measured in the so-called \textit{fill distance}.

Also corresponding \textit{inverse statements} have been derived, 
which allow to conclude the smoothness of the function that is approximated based on the obtained convergence rate \cite{HangelbroekEtAl2018Inverse,schaback2002inverse,ward2012lp}.
Recently \cite{wenzel2025sharp} proved sharp inverse estimates for kernel interpolation in bounded domains using the $L_2$-norm,
and thus deriving a one-to-one correspondence between smoothness and approximation rate.

In this paper, we build on these recent theoretical findings and propose a novel meshfree smoothness detection algorithm, which we call the Sobolev Algorithm for Local Smoothness Analysis (SALSA), which is capable of detecting irregularities and also identifying their order. 
For this we consider the \textit{escaping the native space} regime,
i.e.\ the case where the target function with its irregularities is less smooth than the kernel.
Unlike \cite{romani2019edge}, SALSA focuses on local computations and refinements rather than global reconstructions followed by individual coefficient analysis. 
More precisely, our approach assigns the local smoothness class of the function by taking into account an error decay linked to nested sequences of kernel-based interpolants. Novel results obtained both for direct and inverse statements provide robust theoretical guarantees for the proposed algorithm; Table \ref{tab:summer_results} provides an at-a-glance summary of these theoretical advancements.

\begin{table}[h]
\centering
\caption{Overview of theoretical contributions.}
\label{tab:summer_results}
\resizebox{\textwidth}{!}{
\begin{tabular}{
>{\centering\arraybackslash}p{0.8cm}  %
p{3.5cm} 
>{\centering\arraybackslash}p{3.5cm} 
>{\centering\arraybackslash}p{3.5cm} 
p{5cm}
}
\toprule
& \textbf{Result Type} & \textbf{Baseline Result} & \textbf{New Result} & \textbf{Description} \\
\midrule
\multirow{2}{*}{
  \raisebox{-2.2ex}[0pt][0pt]{\rotatebox[origin=c]{90}{$L_p$}}
} 
& Direct Statement
& \Cref{th:error_estimate_L2} from \cite{narcowich2006sobolev}
& \Cref{th:error_estimate_Lq}
& Generalizes the original $L_2$-based bound to arbitrary $L_p$ spaces. \\

& Inverse Statement
& \Cref{thm:l2-inverse-statement} from \cite{wenzel2025sharp}
& \Cref{thm:l2-inverse-statement_strengthened}
& Provides an improved $L_2$ inverse statement by weakening the required assumptions. \\
\addlinespace
\midrule
\multirow{2}{*}{
  \raisebox{-2.2ex}[0pt][0pt]{\rotatebox[origin=c]{90}{RKHS}}
} 
& Direct Statement 
& --- 
& \Cref{th:direct_statement_RKHS}
& Provides a bound on the growth of the RKHS norm. \\

& Inverse Statement
& --- 
& \Cref{thm:RKHS-inverse-statement}
& Extends the inverse analysis to the RKHS norm case. \\
\bottomrule
\end{tabular}
}
\end{table}

The paper is organized as follows:
\Cref{sec:kernel_interpolation} reviews the necessary background information on kernel interpolation.
\Cref{sec:direct_inverse} extends direct statements and provides new and stronger inverse statements.
\Cref{sec:salsa} then introduces the Sobolev Algorithm for Local Smoothness Analysis (SALSA), 
and its use is subsequently numerically illustrated in \Cref{sec:num_results}.
\Cref{sec:conclusion_outlook} concludes the paper and provides an outlook.

\section{Kernel Interpolation}
\label{sec:kernel_interpolation}

Let $ \Omega \subseteq \Rbb^d $ be a bounded domain and \(X=\{{\bs x}_1,\ldots,
{\bs x}_N\}\subset\Omega\) be a set of data sites with cardinality $ N = \#X$. 
Moreover, let $f \in C(\Omega) $ be a target function generating corresponding target data
$ f_1, \dots, f_{N}$. 
Associated to the set of sites $ X $ there is the
\emph{fill-distance} $h_{X, \Omega}$ of $ X $ in $ \Omega $
and the separation radius $q_X$ of $X$ defined as
\begin{align}
\begin{aligned}
\label{eq:fill_sep_dist}
    h_{X,\Omega} &\isdef \sup_{{\bs x} \in \Omega} \min_{{\bs x} \in X}
    \| {\bs x} - {\bs x}_j \|_2, \\
    q_X &\isdef \frac{1}{2} \min_{{\bs x}_i \neq {\bs x}_j \in X} \Vert {\bs x}_i - {\bs x}_j \Vert_2.
\end{aligned}
\end{align}
We call a sequence of sets $(X_n)_{n \in \N} \subset \Omega$ \emph{quasi-uniform}, if the uniformity constant $\rho_{X_n, \Omega} := h_{X_n, \Omega} / q_{X_n} > 0$ is uniformly bounded, i.e. 
\begin{align*}
\sup_{n \in \N} \rho_{X_n, \Omega} < \infty.
\end{align*}
Our first step is to interpolate or approximate the target function $f$ by means of a kernel using only the given data
\(({\bs x}_1,f_1),\ldots,({\bs x}_N,f_N)\).
For this we will make use of a popular class of kernels, namely \emph{radial basis function} (RBF) kernels.
A function \(\Phi\colon\Rbb^d\to\Rbb\) is said to be \emph{radial}, if and only if there
exists a univariate function \(\phi\colon[0,\infty)\to\Rbb\) such that
\(\Phi({\bs x})=\phi(\|{\bs x}\|_2)\). 
The radial function \(\Phi\) is \emph{strictly positive definite}, if and only if
\[
{\bs\alpha}^\intercal[\Phi({\bs\xi}_i-{\bs\xi}_j)]_{i,j=1}^n{\bs\alpha}>0
\quad\text{for all }
{\bs\alpha}\in\Rbb^n\setminus\{{\bs 0}\},
\]
for any choice of pairwise distinct points
\({\bs\xi}_1,\ldots,{\bs\xi}_n\), \(n\in\Nbb\). In this case, the kernel $K\colon\R^d \times \R^d \rightarrow \R$ defined as
\[
K({\bs x_i},{\bs x_j})
\isdef\Phi({\bs x}_i-{\bs x}_j)
\] 
is the reproducing kernel of a uniquely determined Reproducing Kernel
Hilbert Space (RKHS), its \emph{native space}, that we denote by
\(\big(\Ncal_\Phi,\langle\cdot,\cdot\rangle_{\Ncal_\Phi}\big)\).
In this context, we will use the function \(\Phi\) and its associated
kernel \(K\) synonymously.

Under the notion of finitely smooth kernels, one frequently considers kernels whose Fourier transform exhibits algebraic asymptotics, i.e.
\begin{align}\label{eq:algebraicDecaySobolevSpline}
c_\Phi (1 + \| \boldsymbol{\omega} \|_2^2)^{-\tau}
\leq 
\widehat{\Phi}(\boldsymbol{\omega}) 
\leq C_\Phi
(1 + \| \boldsymbol{\omega} \|_2^2)^{-\tau}, 
\quad \boldsymbol{\omega} \in \mathbb{R}^d,
\end{align}
for some positive constants $0 < c_\Phi \leq C_\Phi < \infty$.
Based on \eqref{eq:algebraicDecaySobolevSpline}, one can conclude
that the native space of the corresponding kernel
is norm-equivalent to the classical Sobolev space
$W_2^{\tau}(\mathbb{R}^d) $, see
\cite[Theorem 6.13]{wendland2004scattereddata} for example.
This norm-equivalence extends to bounded domains $\Omega \subset \R^d$ under mild conditions, 
such as a Lipschitz boundary. 

In the following sections, we will focus on such finitely smooth kernels on bounded Lipschitz domains. 
We collect these assumptions as follows:
\begin{assumption}
\label{ass:kernel_domain}
Let $\Omega \subset \R^d$ be a compact Lipschitz region $\Omega \subset \mathbb{R}^d$ and let $\Phi$ be a continuous kernel such that $\Ncal_\Phi(\Omega) \simeq W_2^\tau(\Omega)$ for some $\tau > d/2$. 
\end{assumption}

Popular examples for such strictly positive RBFs are the \emph{Mat\'ern kernels} or \emph{Sobolev splines} 
$\Phi:=\Phi_{{\tau-d/2}}\colon\Rbb^d \to \Rbb $, dependent on a hyper-parameter 
$ {\tau} > d/2 $. 
These kernels are defined via the radial function
$\phi_{{\tau-d/2}}(r) = \frac{2^{1-(\tau-d/2)}}{\Gamma({\tau-d/2})} r^{{\tau-d/2}} K_{{\tau-d/2}} (r)$, $r \geq 0$
where $ \Gamma $ is the Riemann Gamma function and $ K_{{\tau-d/2}} $ is the
modified Bessel function of the second kind, see \cite{MAT} for example. 
The kernel interpolant $I_X f$ of $f$ in $X$ can be written as
\[
(I_X f)(\bs x) = \sum_{i=1}^{N}\alpha_{i}\Phi({\bs x}-{\bs x}_i), \quad \bs x \in \Omega,
\]
with coefficients $\{\alpha_1,\dots, \alpha_n\} \subset \R$ determined by the linear equation system ${I_X f}(x_i) = f_i$ for all $i=1, ,..., N$.
In particular, the interpolant belongs to the trial space
\begin{equation*}
{I_X f} \in \Ncal_\Phi(X) \isdef \operatorname{span} 
\{\Phi(\cdot - {\bs x}):{\bs x}\in X\},
\end{equation*}
which is a finite dimensional subspace of the native space
\begin{equation} \label{native_space}
\Ncal_\Phi(\Omega):=\overline{\operatorname{span} 
\{\Phi(\cdot - {\bs x}):{\bs x}\in \Omega\}}^{\Vert \cdot \Vert_{\Ncal_\Phi(\Omega)}},
\end{equation}
and ${I_X f}$ converges towards the target function $f \in \Ncal_\Phi(X)$ as $ h_{X,\Omega} \to 0$
under mild assumptions on the kernel and the function \cite{iske2025convergence}.

In order to quantify the error between ${I_X f}$ and $f$,
we recall the following \textit{direct statement},
that uses the uniformity constant $\rho_{X, \Omega} = h_{X, \Omega} / q_X$.
The statement is a special case from \cite[Theorem 4.2]{narcowich2006sobolev},
where the slight improvement in the assumptions is due to \cite[Theorem 4.1]{arcangeli2007extension}:

\begin{theorem}[Direct statement $W_2^\mu(\Omega)$]
\label{th:error_estimate_L2}
Under \Cref{ass:kernel_domain},
let $f \in W_2^\beta(\Omega) \supseteq W_2^\tau(\Omega)$ for some $\beta$ such that $d/2 < \beta \leq \tau$.
Then it holds
\begin{align*}
\Vert f - {I_X f} \Vert_{W_2^\mu(\Omega)} \leq C h_{X, \Omega}^{\beta-\mu} \rho_{X, \Omega}^{\tau-\beta} \Vert f \Vert_{W_2^\beta(\Omega)}, \qquad \qquad 0 \leq \mu \leq \lfloor \beta \rfloor - d/2.
\end{align*}
\end{theorem}
Note that the exponent of $\rho_{X, \Omega}$ is $\tau - \beta$,
while \cite{narcowich2006sobolev} stated $\tau - \mu$, which is a typo.

We remark that a key feature of the previous result lies in the estimation of the convergence of the $L_2(\Omega)$-error (i.e.\ $\mu = 0$) 
when the underlying function does not belong to the native space of the kernel, i.e., it is less smooth.
\Cref{th:error_estimate_L2} provides escaping the native space bounds only for Hilbert space norms (i.e.\ $p=2$).
We extend these bounds in \Cref{th:error_estimate_Lq} to more general $p$-norms.

These direct estimates establish sufficient conditions for convergence: 
for $f \in W_2^\beta(\Omega)$, the interpolant using quasi-uniformly distributed points $X$ satisfies $\|f - {I_X f}\|_{L_2(\Omega)} \leq C h_{X,\Omega}^\beta$. 
In contrast, inverse theorems prove the contrary:
If $\|f - {I_X f}\|_{L_2(\Omega)} \leq Ch_{X,\Omega}^\beta$ holds uniformly for some $\beta > 0$, 
then necessarily $f \in H^{\beta - \varepsilon}(\Omega)$ for any $\varepsilon > 0$.
This is made precise in the following theorem, taken from \cite[Theorem 1]{wenzel2025sharp}:
\begin{theorem}[Inverse statement]
\label{thm:l2-inverse-statement}
Under \Cref{ass:kernel_domain}, 
consider $f \in C(\Omega)$ such that for some constants $c_f, \beta, h_0 > 0$ it holds
\begin{equation}\label{eq:l2-inverse-statement}
\|f - {I_X f}\|_{L_2(\Omega)} \leq c_f h_{X,\Omega}^\beta
\end{equation}
for all $X \subset \Omega$ with $h_{X,\Omega} \leq h_0$ and  $\rho_{X, \Omega} \leq 44$.
If $\beta \in (0, \tau]$, then $f \in H^{\beta'}(\Omega)$ for all $\beta' \in (0, \beta)$. If $\beta > \tau$, then $f \in \Ncal_\Phi(\Omega) \simeq W_2^\tau(\Omega)$.
\end{theorem}

Thus, 
\Cref{th:error_estimate_L2} and \Cref{thm:l2-inverse-statement} together yield a complete one-to-one correspondence between the decay rate and the smoothness,
thus establishing a powerful analytical tool for smoothness detection.

Nevertheless, we further strengthen the inverse result of \Cref{thm:l2-inverse-statement} by weakening its assumptions, which yields \Cref{thm:l2-inverse-statement_strengthened}.
This is made possible by the recent results \cite[Theorem 4.9]{zhengjie2025inverse},
which states the Bernstein inequality
\begin{align}
\label{eq:bernstein_old}
\Vert u \Vert_{H^\beta(\Omega)} \leq C q_{X}^{-\beta} \Vert u \Vert_{L_2(\Omega)}
\end{align}
for some constant $C = C_{d, \kappa, \tau, \Omega} > 0$ and any trial functions $u \in \Ncal_\Phi(X)$ for $X \subset \Omega$,
as soon as $\beta$ satisfies $d/2 < \beta \leq \tau$ or $0 \leq \beta \leq \lfloor \tau \rfloor$.
In \Cref{th:bernstein}, 
we will generalize this Bernstein inequality to any $0 \leq \beta \leq \tau$,
and then later on leverage it to prove \Cref{thm:l2-inverse-statement_strengthened}.

\section{Direct and inverse statements}
\label{sec:direct_inverse}

In this section, we provide some theoretical improvements on the direct and inverse statements stated before.
First, we extend the direct statement error bounds of \Cref{th:error_estimate_L2} from the $L_2(\Omega)$ norm to general $L_p(\Omega)$ norms,
which is formulated in \Cref{th:error_estimate_Lq}.
Second, we strengthen the inverse statement of \Cref{thm:l2-inverse-statement} by weakening its assumptions, 
which is formulated in \Cref{thm:l2-inverse-statement_strengthened}.
Finally we also provide inverse statements using the $\Ncal_\Phi(\Omega)$-norm,
which is given in \Cref{thm:RKHS-inverse-statement}.

As both the two inverse statements rely on an approximation using well distributed points,
we formulate this assumption on the point sequence in the following assumption:

\begin{assumption}
\label{ass:points}
Let $(X_n)_{n \in \N} \subset \Omega$ be a nested sequence of quasi-uniformly distributed point sets with geometrically decaying fill distance,
i.e.\ it holds 
\begin{align}
\label{eq:assumption_decay_fill_dist}
c_0' a^n \leq q_{X_n} \leq h_{X_n, \Omega} \leq c_0 a^n
\end{align}
for constants $c_0, c_0' > 0$ and $a \in (0, 1)$.
\end{assumption}

Note that in the following, we do not track constants explicitly, in particular constants within proofs may change from line to line.

\subsection{Extended direct statement}

We start by generalizing \Cref{th:error_estimate_L2} from $W_2^\mu(\Omega)$ norms to more general $W_q^\mu(\Omega)$ norms.
The idea is to provide direct statements using non-Hilbert norms,
and later complement them with the corresponding inverse statements 
that we were, however, not able to obtain, see the discussion at the end of \Cref{subsec:strengthened_invers}.

\begin{theorem}[Direct statement $L_q(\Omega)$]
\label{th:error_estimate_Lq}
Under \Cref{ass:kernel_domain},
consider $\beta$ such that $\tau \geq \beta \geq \lfloor \beta \rfloor > d/2$.
Then it holds for every $f \in W_2^\beta(\Omega)$ and $q\in [1, \infty]$ that
\begin{align*}
\Vert f - {I_X f} \Vert_{W_q^{\mu}(\Omega)} \leq C h_{X,\Omega}^{\beta-\mu-d(1/2-1/q)_+} \rho_{X, \Omega}^{\tau-\beta} \Vert f \Vert_{W_2^\beta(\Omega)}, \qquad \qquad 0 \leq  \mu < \lfloor \beta \rfloor - d/2
\end{align*}
and in particular ($\mu=0$):
\begin{align*}
\Vert f - {I_X f} \Vert_{L_q(\Omega)} \leq C h_{X,\Omega}^{\beta-d(1/2-1/q)_+} \rho_{X, \Omega}^{\tau - \beta} \Vert f \Vert_{W_2^\beta(\Omega)}.
\end{align*}
\end{theorem}

\begin{proof}
The proof is along the lines of the the proof \cite[Theorem 4.2]{narcowich2006sobolev},
however we apply a more general sampling inequality in the very first step.
The function $f - {I_X f}$ vanishes on $X$, thus by \cite[Theorem 2.2]{gia2006continuous} we have
\begin{align*}
\Vert f - {I_X f} \Vert_{W_q^{m}(\Omega)} \leq C h^{\beta - m - d(1/2 - 1/q)_+} \Vert f - {I_X f} \Vert_{W_2^\beta(\Omega)}, 
\qquad \qquad m \leq \lfloor \beta \rfloor - d/2.
\end{align*}
It remains to estimate the right-hand side norm $\Vert f - {I_X f} \Vert_{W_2^\beta(\Omega)}$,
which can be done exactly as in \cite[Proof of Theorem 4.2]{narcowich2006sobolev} and yields the bound \cite[Eq.~(12)]{narcowich2006sobolev}
\begin{align*}
\Vert f - {I_X f} \Vert_{W_2^\beta(\Omega)} \leq C \rho_{X, \Omega}^{\tau - \beta} \Vert f \Vert_{W_2^\beta(\Omega)}.
\end{align*}
Thus both these inequalities together give 
\begin{align*}
\Vert f - {I_X f} \Vert_{W_q^{m}(\Omega)} \leq C h^{\beta - m - d(1/2 - 1/q)_+} \rho_{X, \Omega}^{\tau - \beta} \Vert f \Vert_{W_2^\beta(\Omega)}.
\end{align*}
Interpolating between $m=0$ and $m=\lfloor \beta \rfloor - d/2$ gives the general result depending on $\mu$. %
\end{proof}

Next we state a kind of direct statement using the $\Ncal_\Phi(\Omega)$ norm.
Since we consider $f \notin \Ncal_\Phi(\Omega)$,
we cannot consider $\Vert f - I_{X_n}f \Vert_{\Ncal_\Phi(\Omega)}$ 
and instead consider $\Vert I_{X_{n}}f - I_{X_{n+1}}f \Vert_{\Ncal_\Phi(\Omega)}$,
for which we provide a bound on its growth.

\begin{theorem}[Direct statement --- $\Ncal_\Phi(\Omega)$-version]
\label{th:direct_statement_RKHS}
Under \Cref{ass:kernel_domain} and \Cref{ass:points}, 
consider $f \in W^\beta_2(\Omega)$.
Then it holds
\begin{align*}
\Vert I_{X_{n}}f - I_{X_{n+1}}f \Vert_{\Ncal_\Phi(\Omega)} \leq C h_{X_n,\Omega}^{\beta - \tau} \Vert f \Vert_{W_2^\beta(\Omega)}
\end{align*}
as well as
\begin{align*}
\Vert I_{X_{n}}f \Vert_{\Ncal_\Phi(\Omega)} \leq C' h_{X_n,\Omega}^{\beta - \tau}
\end{align*}
\end{theorem}

\begin{proof}
Applying the Bernstein inequality in \Cref{th:bernstein} with $\beta=\tau$ to $I_{X_n} f - I_{X_{n+1}} f$
and estimating the right hand side with help of \Cref{th:error_estimate_Lq} gives the statement:
\begin{align*}
\Vert I_{X_{n}}f - I_{X_{n+1}}f \Vert_{\Ncal_\Phi(\Omega)} \leq&~ Cq_{X_{n+1}}^{-\tau} \Vert I_{X_{n}}f - I_{X_{n+1}}f \Vert_{L_2(\Omega)} \\
\leq&~ Cq_{X_{n+1}}^{-\tau} \left( \Vert f - I_{X_n}f \Vert_{L_2(\Omega)} + \Vert f - I_{X_{n+1}}f \Vert_{L_2(\Omega)} \right) \\
\leq&~ Cq_{X_{n+1}}^{-\tau} h_{X_n,\Omega}^\beta \Vert f \Vert_{W_2^\beta(\Omega)} \\
\leq&~ C h_{X_n,\Omega}^{\beta - \tau} \Vert f \Vert_{W_2^\beta(\Omega)}.
\end{align*}
This then yields
\begin{align*}
\Vert I_{X_{n}}f \Vert_{\Ncal_\Phi(\Omega)} \leq&~ \sum_{\ell=2}^n \Vert I_{X_\ell} f - I_{X_{\ell-1}} f \Vert + \Vert I_{X_1} f \Vert_{\Ncal_\Phi(\Omega)} \\
\leq&~ \Vert I_{X_1} f \Vert_{\Ncal_\Phi(\Omega)} + C \sum_{\ell=2}^n h_{X_\ell, \Omega}^{\beta-\tau} \Vert f \Vert_{W_2^\beta(\Omega)} \\
\leq&~ \Vert I_{X_1} f \Vert_{\Ncal_\Phi(\Omega)} + C \sum_{\ell=2}^n a^{\ell(\beta-\tau)} \Vert f \Vert_{W_2^\beta(\Omega)}.
\end{align*}
In the last step, \Cref{ass:points} was applied.
The remaining parts can be further estimated using the geometric sum formula, recalling $a < 1$ and $\beta - \tau > 0$, and reapplying \Cref{ass:points}.
So that we end up with
\begin{align*}
\Vert I_{X_{n}}f \Vert_{\Ncal_\Phi(\Omega)} \leq&~ \Vert I_{X_1} f \Vert_{\Ncal_\Phi(\Omega)} + C h_{X_n, \Omega}^{\beta-\tau} \Vert f \Vert_{W_2^\beta(\Omega)} \\
&\leq C' h_{X_n, \Omega}^{\beta-\tau}.
\end{align*}
\end{proof}

\subsection{Strengthened and new inverse statement}
\label{subsec:strengthened_invers}

We begin this section by slightly generalizing the Berstein inequality \cite[Theorem 4.9]{zhengjie2025inverse} stated in \eqref{eq:bernstein_old},
which was not proven for small gap for the case that $\lfloor \tau \rfloor < d/2$.
This gap can be closed, such that the Bernstein inequality holds for $\beta \in [0, \tau]$:
\begin{theorem}
\label{th:bernstein}
Under \Cref{ass:kernel_domain},
let $\beta \in [0, \tau]$.
Then there exists a constant $C = C_{d, \kappa, \tau, \Omega} > 0$ such that the following Bernstein inverse inequality holds
\begin{align}
\label{eq:bernstein}
\Vert u \Vert_{H^\beta(\Omega)} \leq C q_{X}^{-\beta} \Vert u \Vert_{L_2(\Omega)}
\end{align}
for all point sets $X \subset \Omega$ and all trial functions $u \in \Ncal_\Phi(X)$.
\end{theorem}
\begin{proof}
By Gagliardo-Nirenberg interpolation inequality and subsequent use of \cite[Theorem 4.9]{zhengjie2025inverse} for $\beta = \tau$ we obtain (constants might change from line to line)
\begin{align*}
\Vert u \Vert_{H^\beta(\Omega)} &\leq C \Vert u \Vert_{L_2(\Omega)}^{1-\beta/\tau} \cdot \Vert u \Vert_{H^\tau(\Omega)}^{\beta/\tau} \\
&\leq C \Vert u \Vert_{L_2(\Omega)}^{1-\beta/\tau} \cdot \left( Cq_X^{-\tau} \Vert u \Vert_{L_2(\Omega)} \right)^{\beta/\tau} \\
&= Cq_X^{-\beta} \cdot \Vert u \Vert_{L_2(\Omega)}.
\end{align*}
\end{proof}

The next theorem strengthens \Cref{thm:l2-inverse-statement} by both weakening its assumptions as well as extending from kernel interpolation to general kernel approximations:

The assumption is weakened, as we now only require a single sequence of sets of quasi-uniformly distributed points that provides an error decay instead of the assumption that we are given the error decay for any set of quasi-uniformly distributed points.
This improvement is made possible by utilizing a previously unavailable Bernstein inequality from \cite{zhengjie2025inverse},
which we slightly generalized in \Cref{th:bernstein}.
The use of general kernel approximation methods (such as regularized approximation, least squares or Galerkin) also allows to start with $f \in L_2(\Omega)$ instead of $f \in \mathcal{C}(\Omega)$.

\begin{theorem}[Inverse statement — strengthened $L_2(\Omega)$-version]
\label{thm:l2-inverse-statement_strengthened}
Under \Cref{ass:kernel_domain} and \Cref{ass:points},
let $f \in L_2(\Omega)$ and assume there exists a sequence of point based approximants, i.e.\ $(u_{X_n})_{n \in \N} \subset \Ncal_\Phi(\Omega)$ with $u_{X_n} \in \Ncal_\Phi(X_n)$ such that
\begin{align}
\label{eq:l2-inverse-statement_strengthened}
\Vert f - u_{X_n} \Vert_{L_2(\Omega)} \leq c_f h_{X_n,\Omega}^\beta
\end{align}
holds for some $c_f >0$ and $\beta \in (0, \tau]$.
Then $f \in H^{\beta'}(\Omega)$ for all $\beta' \in (0, \beta)$. 
\end{theorem}

\begin{proof}
We consider the difference of subsequent approximants, 
i.e.\ $u_{X_{n+1}} - u_{X_n}$.
It holds $u_{X_{n+1}} - u_{X_n} \in \Ncal_\Phi(X_{n+1})$ due to the assumption on the nestedness.
Applying the Bernstein inequality Eq.~\eqref{eq:bernstein} using $\beta' < \beta$ gives
\begin{align*}
\Vert u_{X_{n+1}} - u_{X_n} \Vert_{H^{\beta'}(\Omega)} \leq C q_{X_{n+1}}^{-\beta'} \Vert u_{X_{n+1}} - u_{X_n} \Vert_{L_2(\Omega)}.
\end{align*}
Note that here we also made use of the nestedness $X_n \subset X_{n+1}$ of the sets,
such that $q_{X_n \cup X_{n+1}} = q_{X_{n+1}}$.
Inserting a zero and leveraging the assumed decay gives us 
\begin{align*}
\Vert u_{X_{n+1}} - u_{X_n} \Vert_{H^{\beta'}(\Omega)} 
\leq&~ C q_{X_{n+1}}^{-\beta'} \Vert (u_{X_{n+1}} - f) - (u_{X_n} - f) \Vert_{L_2(\Omega)} \\
\leq&~ C q_{X_{n+1}}^{-\beta'} \cdot \left( c_f h_{X_{n+1}, \Omega}^\beta + c_f h_{X_n, \Omega}^\beta \right) \\
\leq&~ 2c_f C q_{X_{n+1}}^{-\beta'} h_{X_{n}, \Omega}^\beta \\
\leq&~ 2c_f C c_0'^{-\beta'} c_0^\beta a^{-\beta'} a^{n(\beta-\beta')}.
\end{align*}
Now we may conclude by following the remaining proof of \cite[Theorem 1]{wenzel2025sharp},
i.e.\ we show that $( u_{X_n} )_{n \in \N}$ is a Cauchy sequence in $H^{\beta'}(\Omega)$ and conclude a limiting element by the completeness of $H^{\beta'}(\Omega)$.
Let $n > m > m_0$, then
\begin{align}
\label{eq:cauchy_sequence_calculation}
\Vert u_{X_{n}} - u_{X_m} \Vert_{H^{\beta'}(\Omega)} &= \left \Vert \sum_{\ell=m}^{n-1} u_{X_{\ell+1}} - u_{X_{\ell}} \right \Vert_{H^{\beta'}(\Omega)} \notag \\
&\leq \sum_{\ell=m}^{n-1} \Vert u_{X_{\ell+1}} - u_{X_{\ell}} \Vert_{H^{\beta'}(\Omega)} \notag \\
&\leq 2c_f C c_0'^{-\beta'} c_0^\beta a^{-\beta'} \cdot \sum_{\ell=m}^{\infty} a^{\ell (\beta-\beta')} 
\leq \sqrt{C''} \cdot \frac{a^{(\beta-\beta') m_0}}{1 - a^{(\beta-\beta')}} \stackrel{m_0 \rightarrow \infty}{\longrightarrow} 0.
\end{align}
In the final step, $\beta' < \beta$ and $a \in (0, 1)$ was used.
Since $H^{\beta'}(\Omega)$ is a complete space, there exists a unique limiting element $\tilde{f} \in H^{\beta'}(\Omega)$ such that $u_{X_n} \rightarrow \tilde{f}$.
Actually we have $\tilde{f} = f$ due to
\begin{align}
\begin{aligned}
\label{eq:estimate_to_show_f_equal_ftilde}
\Vert f - \tilde{f} \Vert_{L_2(\Omega)} 
&\leq \Vert f - u_{X_n} \Vert_{L_2(\Omega)} + \Vert \tilde{f} - u_{X_n} \Vert_{L_2(\Omega)} \\
&\leq c_f h_{X_n, \Omega}^{\beta} + C_{H^{\beta'}(\Omega) \hookrightarrow L_2(\Omega)} \Vert \tilde{f} - u_{X_n} \Vert_{H^{\beta'}(\Omega)} \stackrel{n \rightarrow \infty}{\longrightarrow} 0,
\end{aligned}
\end{align}
and thus it holds $f \in H^{\beta'}(\Omega)$ for any $\beta' < \beta$.
\end{proof}

Next,
we prove a novel inverse statement, that assumes a bound on the approximants measured in the $\Vert \cdot \Vert_{\Ncal_\Phi(\Omega)}$ norm.
Since the function $f$ of interest is usually not included in the RKHS, i.e.\ $f \notin \Ncal_\Phi(\Omega)$,
the assumed bound is on $\Vert I_{X_{n+1}} f - I_{X_n} f \Vert_{\Ncal_\Phi(\Omega)}$,
as the approximants are included in the RKHS.

\begin{theorem}[Inverse statement --- $\Ncal_\Phi(\Omega)$-version]
\label{thm:RKHS-inverse-statement}
Under \Cref{ass:kernel_domain} and \Cref{ass:points}, 
let $f \in \mathcal{C}(\Omega)$ and assume the sequence of kernel interpolants $(I_{X_n}f)_{n \in \N} \subset \Ncal_\Phi(\Omega)$ satisfies 
\begin{align}
\label{eq:RKHS-inverse-statement}
\Vert I_{X_{n+1}}f - I_{X_n}f \Vert_{\Ncal_\Phi(\Omega)} \leq c_f h_{X_n,\Omega}^{\beta - \tau}
\end{align}
for some $c_f >0$ and $\beta \in (0, \tau]$. Then $f \in H^{\beta'}(\Omega)$ for all $\beta' \in (0, \beta)$. 
\end{theorem}

\begin{proof}
Note that in the following, the constant $C$ might change from line to line.
Standard sampling inequalities (see e.g.\ \cite[Theorem 2.2]{gia2006continuous}) applied to $I_{X_{n+1}}f - I_{X_n}f$ (which has zeros in $X_n$) gives
\begin{align*}
\Vert I_{X_{n+1}}f - I_{X_n}f \Vert_{L_2(\Omega)} 
&\leq C h_{X_n,\Omega}^\tau \Vert I_{X_{n+1}}f - I_{X_n}f \Vert_{W_2^\tau(\Omega)} \\
&\leq C h_{X_n,\Omega}^\tau \Vert I_{X_{n+1}}f - I_{X_n}f \Vert_{\Ncal_\Phi(\Omega)} \\
&\leq C h_{X_n,\Omega}^{\beta},
\end{align*}
where in the last inequality the assumed bound from Eq.~\eqref{eq:RKHS-inverse-statement} was applied.
Now we would like to bound $\Vert I_{X_{n+1}}f - I_{X_n}f \Vert_{H^{\theta \tau}(\Omega)}$,
which can be done by interpolation theory (i.e.\ using standard Gagliardo-Nirenberg inequality, see e.g.\ \cite[Theorem 1]{brezis2018gagliardo}). 
This and inserting the rates from before gives
\begin{align*}
\Vert I_{X_{n+1}}f - I_{X_n}f \Vert_{H^{\theta \tau}(\Omega)} 
&\leq C \Vert I_{X_{n+1}}f - I_{X_n}f \Vert_{L_2(\Omega)}^{1-\theta} \cdot \Vert I_{X_{n+1}}f - I_{X_n}f \Vert_{\Ncal_\Phi(\Omega)}^\theta \\
&\leq C h_{X_n, \Omega}^{\beta(1-\theta)} \cdot h_{X_{n+1},\Omega}^{(\beta - \tau)\theta} \leq C h_{X_n,\Omega}^{\beta(1-\theta)} \cdot h_{X_n,\Omega}^{(\beta-\tau)\theta} \\
&= C h_{X_n,\Omega}^{\beta- \tau \theta} \leq C a^{n (\beta- \tau \theta)}.
\end{align*}
The final estimate used the assumption on the point distribution from \Cref{ass:points}, in particular recall that $a \in (0, 1)$.

For $\theta < \frac{\beta}{\tau} \Leftrightarrow \theta \tau < \beta$,
the exponent $\beta - \tau \theta$ is positive.
Using standard arguments, we may thus show that for $\theta \in (0, \beta/\tau) \subset (0, 1)$,
the sequence $(I_{X_n}f)_{n \in \N} \in H^{\theta \tau}(\Omega) \supset \Ncal_\Phi(\Omega)$ is a Cauchy sequence in $H^{\theta \tau}(\Omega)$:
Let $n > m > m_0$,
\begin{align*}
\Vert I_{X_{n}}f - I_{X_m}f \Vert_{H^{\theta \tau}(\Omega)} &= \left \Vert \sum_{\ell=m}^{n-1} I_{ X_{\ell+1}}f - I_{X_{\ell}}f \right \Vert_{H^{\theta \tau}(\Omega)} \\
&\leq \sum_{\ell=m}^{n-1} \Vert I_{X_{\ell+1}}f - I_{ X_{\ell}}f \Vert_{H^{\theta \tau}(\Omega)} \\
&\leq C \cdot \sum_{\ell=m}^{\infty} a^{n (\beta - \tau \theta)}  
\leq C \cdot \frac{a^{m_0 (\beta - \tau \theta)}}{1 - a^{\beta- \tau \theta}} \stackrel{m_0 \rightarrow \infty}{\longrightarrow} 0.
\end{align*}
Since $H^{\theta \tau}(\Omega)$ is a complete space, there exists a unique limiting element $\tilde{f} \in H^{\theta \tau}(\Omega)$ such that $I_{X_n}f \rightarrow \tilde{f}$.
With the same argumentation as in Eq.~\eqref{eq:estimate_to_show_f_equal_ftilde} we obtain $f = \tilde{f}$.

\end{proof}

Note that a bound $\Vert I_{X_n} f \Vert_{\Ncal_\Phi(\Omega)} \leq c_f' h_{X_n, \Omega}^{\beta-\tau}$ for all $n \in \N$ already implies Eq.~\eqref{eq:RKHS-inverse-statement} via triangle inequality
$\Vert I_{X_{n+1}}f - I_{X_n}f \Vert_{\Ncal_\Phi(\Omega)} \leq \Vert I_{X_{n+1}}f \Vert_{\Ncal_\Phi(\Omega)} + \Vert I_{X_n}f \Vert_{\Ncal_\Phi(\Omega)}$.

Finally we conclude this section with a small comment on sharp inverse statements using non-Hilbert $L_p(\Omega)$, $p \neq 2$, convergence rates.
A preliminary statement in this direction using $p=\infty$ was given in \cite[Theorem 6.1]{schaback2002inverse},
which was however not sharp compared to the corresponding direct statement formulated e.g.\ in \Cref{th:error_estimate_Lq}. 
Whether a sharp inverse statement using $L_\infty(\Omega)$ can be obtained,
i.e.\ whether a decay assumption as $\Vert f - u_{X_n} \Vert_{L_\infty(\Omega)} \leq c_f h_{X,\Omega}^{\beta-d/2}$ actually implies $f \in W_2^\beta(\Omega)$,
remains unclear.
In fact, \cite{ward2012lp} deals with inverse statements using $L_p(\R^d)$ norms 
(i.e.\ working on $\R^d$ instead of in compact regions $\Omega \subset \R^d$),
but concludes that the target function is in some Bessel potential space,
which does not coincide with the Sobolev space $W_2^\beta(\Omega)$.

\section{Steps of SALSA}
\label{sec:salsa}

\subsection{Theoretical rationale}

The direct and inverse statements of \Cref{th:error_estimate_Lq} and \Cref{thm:l2-inverse-statement_strengthened} (using $L_2(\Omega)$ norm)
and the respective \Cref{th:direct_statement_RKHS} and \Cref{thm:RKHS-inverse-statement} (using $\Ncal_\Phi(\Omega)$ norm)
establish a one-to-one correspondence between convergence rate and smoothness of a function.
Thus in practice, the smoothness of a function can be obtained by measuring the rate of decay respective growth of the $L_2(\Omega)$ respective $\Ncal_\Phi(\Omega)$ norm of subsequent approximants.
Due to extension and restriction statements for native space \cite[Section 10.7]{wendland2004scattereddata},
this principle can directly be applied in a localized way.
This thus allows to calculate the local smoothness based on the convergence rates of local approximants.

\subsection{The SALSA algorithm}

Based on the theoretical rationale explained before, SAL\-SA estimates the local Sobolev regularity of $f$ on a region 
$O$ by observing how kernel interpolants on nested point sets converge as the fill distance shrinks.

Let $X_1 \subset X_2 \subset \dots \subset X_M \subset O$ be a sequence of nested sampling sets with geometrically decaying fill distances $h_{X_m, \Omega}$, as specified in \Cref{ass:points}. For each level $m$, we construct a kernel interpolant $I_{X_m}f$ from the function samples $\bs f_m = \{f(\bs{x}) : \bs{x} \in X_m\}$.
To estimate the local regularity of the function $f$, we compute two sequences of quantities: the native-space norm of each interpolant, $\left\|I_{X_m} f\right\|_{\mathcal{N}_{\Phi}(X_m)}$, and the squared discrete $L^2$-norm of the difference between successive interpolants evaluated at the finer scale $X_M$, i.e.
$\left\|I_{X_m} f-I_{X_{m-1}} f\right\|_{L^2\left(X_M\right)}^2$.

Each sequence is fitted to a power law using log-log regression. The $L_2$-based measure directly yields an estimate $\tilde{\beta}_2$ of the Sobolev smoothness $\beta$, while the native-space norm yields an estimate $\tilde{\gamma}_{\mathcal{N}}$ of $\gamma = \beta - \tau$, where $\tau$ is the smoothness of the reproducing kernel. This procedure is formally described in Algorithm~\ref{alg:local_smoothness_detection}.

\begin{algorithm}[htbp!]
\caption{Sobolev Algorithm for Local Smoothness Analysis (SALSA)}
\label{alg:local_smoothness_detection}
\begin{algorithmic}[1]
\Require Nested point sets
\(
X_1 \subset X_2 \subset \cdots \subset X_M \subset \Omega
\)
\Require Function values
\(
\bs{f}_m = \{ f(\bs{x}) : \bs{x} \in X_m \}, \quad m=1,\ldots,M
\)
\Ensure Estimated local Sobolev smoothness
\(
\beta
\) of 
\(
f
\).
\vspace{0.7em}
\State \underline{Initialize:}
\(
\quad \bs{c}^{(2)} \gets [~],
\quad \bs{c}^{(\mathcal{N})} \gets [~]
\)
\vspace{0.7em}
\For{\( m = 2 \) to \( M \)}
    \vspace{0.7em}
    \State \underline{Compute:}
    \(
    \quad \bs{K}_{m-1} = \bigl( K(\bs{x}_i, \bs{x}_j) \bigr)_{\substack{\bs{x}_i \in X_{m-1} \\ \bs{x}_j \in X_{m-1}}} , 
    \quad \bs{K}_{m} = \bigl( K(\bs{x}_i, \bs{x}_j) \bigr)_{\substack{\bs{x}_i \in X_{m} \\ \bs{x}_j \in X_{m}}} 
    \)
    \vspace{0.7em}
    
    \State \underline{Solve:}
    \(
    \quad \bs{K}_{m-1} \bs{\alpha}_{m-1} = \bs{f}_{m-1},
    \quad \bs{K}_{m} \bs{\alpha}_{m} = \bs{f}_{m}
    \)
    \vspace{0.7em}
    
    \State \underline{Evaluate on $X_M$:}
    \(
    \quad \bs I_{X_{m-1}} f = \bs{K}_{M,m-1} \bs{\alpha}_{m-1}, \quad
    \bs I_{X_{m}} f = \bs{K}_{M,m} \bs{\alpha}_{m}
    \)
    
    \vspace{0.7em}
    \State \underline{Compute:}
    \(
    \quad c^{(2)}_{m-1} = \left\| \bs I_{X_{m}}f - \bs I_{X_{m-1}}f \right\|_{L_2(X_M)}^2, 
    \) \\
        \vspace{0.3em}
    \(
    \quad \quad \quad \quad \quad \quad  \quad  c^{(\mathcal{N})}_{m-1} = \bs{f}_{m-1}^\top \bs{K}_{m-1}^{-1}  \bs{f}_{m-1}
    \) 
    
    \vspace{0.7em}
    
    \State $\bs c^{(2)} \gets [\bs c^{(2)}, c^{(2)}_{m-1}] $
        \vspace{0.3em}
        
    \State $\bs c^{(\mathcal{N})} \gets [\bs c^{(\mathcal{N})}, c^{(\mathcal{N})}_{m-1}] $
\vspace{0.7em}
\EndFor
\vspace{0.7em}
\State \underline{Fit:}
\(
\quad \bs c^{(2)} \sim h_{X_m,\Omega}^{2 \tilde{\beta}_2}, 
\quad
\bs c^{(\mathcal{N})} \sim h_{X_m,\Omega}^{2 \tilde{\gamma}_{\mathcal{N}}}
\)
\vspace{0.7em}
\State \underline{Estimate:} 
\(
\quad \beta = \tilde{\beta}_2, 
\quad \text{or}
\quad \beta = \tau + \tilde{\gamma}_{\mathcal{N}}
\)
\Comment{See Assumption \ref{ass:kernel_domain} for $\tau$.}
\vspace{0.7em}
\State \Return 
\(
\beta
\)
\end{algorithmic}
\end{algorithm}

\subsection{Nested nodal subsets construction} \label{sec:subsets_section}

Before delving into the numerical validation of SALSA, 
we discuss in the following two approaches to generate nested sequences of nodes that agree with \Cref{ass:points}. 
The first approach relies on a theoretical framework in which we assume knowledge of the underlying function on the whole domain. In contrast, with the second approach the target function is only known at a given set of data sites.

\subsubsection{Closed-form function}\label{sec:stecntils_section} When the function is known in closed-form everywhere on the domain, we adopt fixed translation-invariant stencils. The basic idea is to precompute a family of nested reference stencils of increasing resolution $\{ \mathcal{S}_m\}_{m=1}^M$, and then reuse them for every evaluation \textit{center} point $\bs x_c \subset X$ by means of a simple translation and dilation. The radius of the reference stencil is chosen according to the dataset’s fill distance $h_{X,\Omega}$, so that:
\[
\mathcal{S}_m \subset [\bs x_c - h_{X,\Omega}, \bs x_c + h_{X,\Omega}]^d.
\]
Precisely, the nested stencils are constructed by successive midpoint insertions, so that the number of nodes per edge $ \mathcal{E} \isdef \mathcal{S}^{1/d}$ satisfies the recurrence formula
\[
\# \mathcal{E}_{m+1} = 2(\# \mathcal{E}_{m}) - 1,
\]
where $\# \mathcal{A}$ denotes the cardinality of the set $\mathcal{A}$. A visualization of the stencils in one-, two-, and three-dimensional settings is presented in Figure~\ref{fig:local_stencils}. In the experiments, we refer to this approach as the \emph{Fixed Stencils Method}.

\begin{remark}
When evaluating points near the boundary of the domain, the stencils are generated in the same manner, but only the nodes lying inside the domain are retained. In other words, the translation-invariant reference stencils are clipped to the domain before use at boundary centers.
\end{remark}

\begin{figure}[]
    \centering
    \begin{subfigure}{0.8\textwidth}
        \centering
        \includegraphics[width=0.4\linewidth]{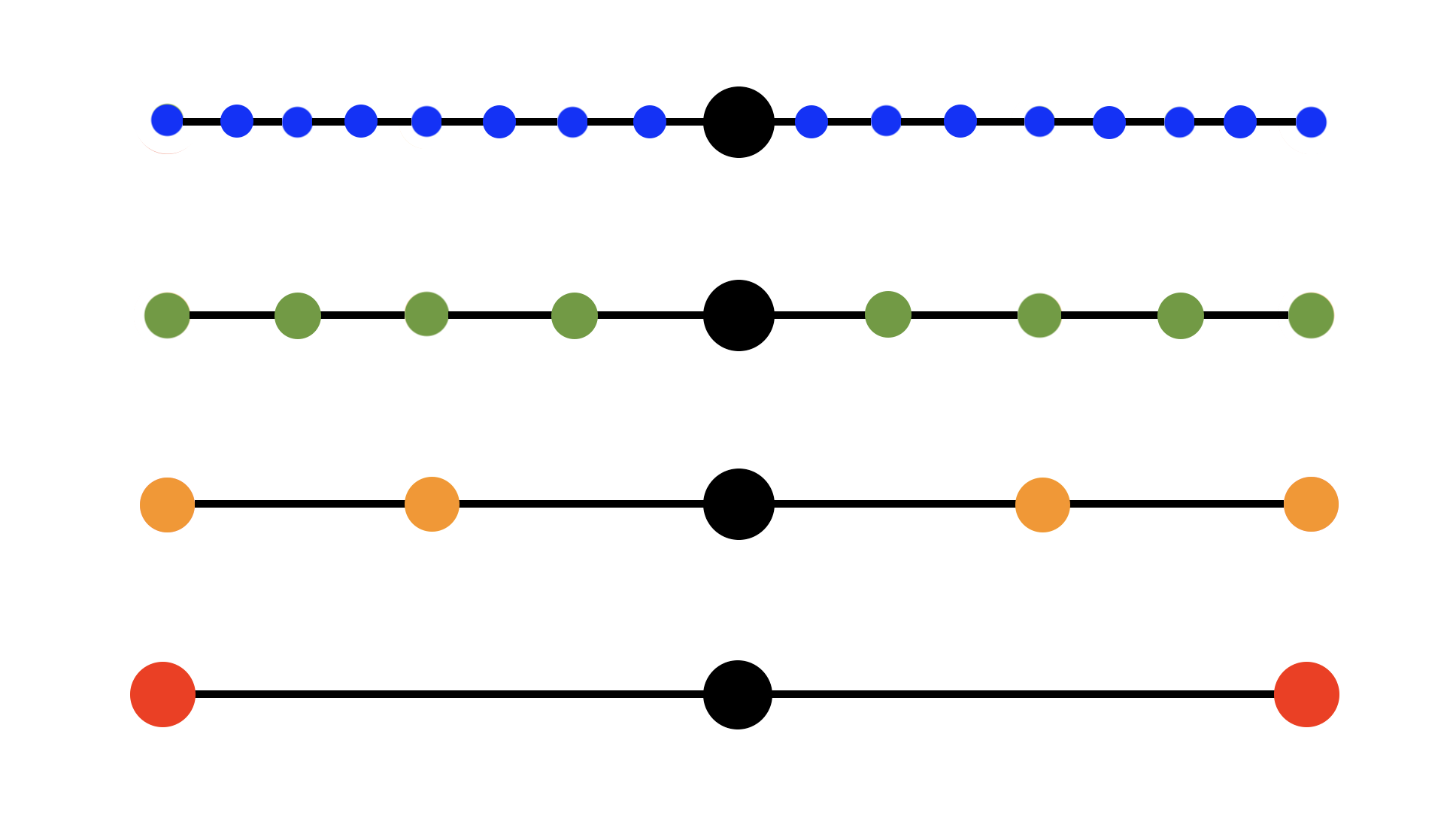} %
        \caption{1D stencils for $m=4$ refinements.}
        \label{fig:stencil_1d}
    \end{subfigure}
    
    \vspace{1em} %

    \begin{subfigure}{0.3\textwidth}
        \centering
        \includegraphics[width=0.54\linewidth]{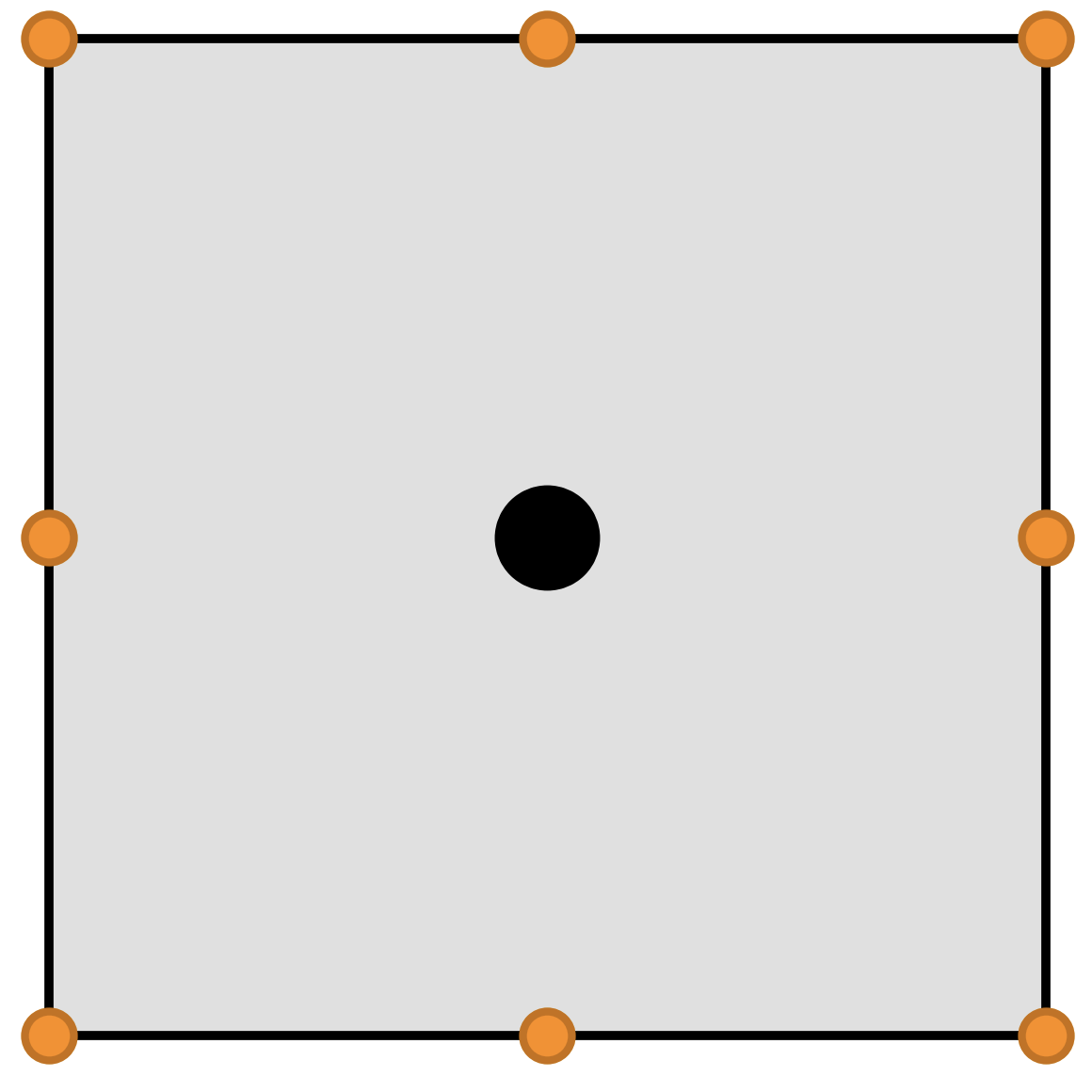}
        \caption{2D stencil for $m$.}
        \label{fig:stencil_2d_2}
    \end{subfigure}
    \begin{subfigure}{0.3\textwidth}
        \centering
        \includegraphics[width=0.54\linewidth]{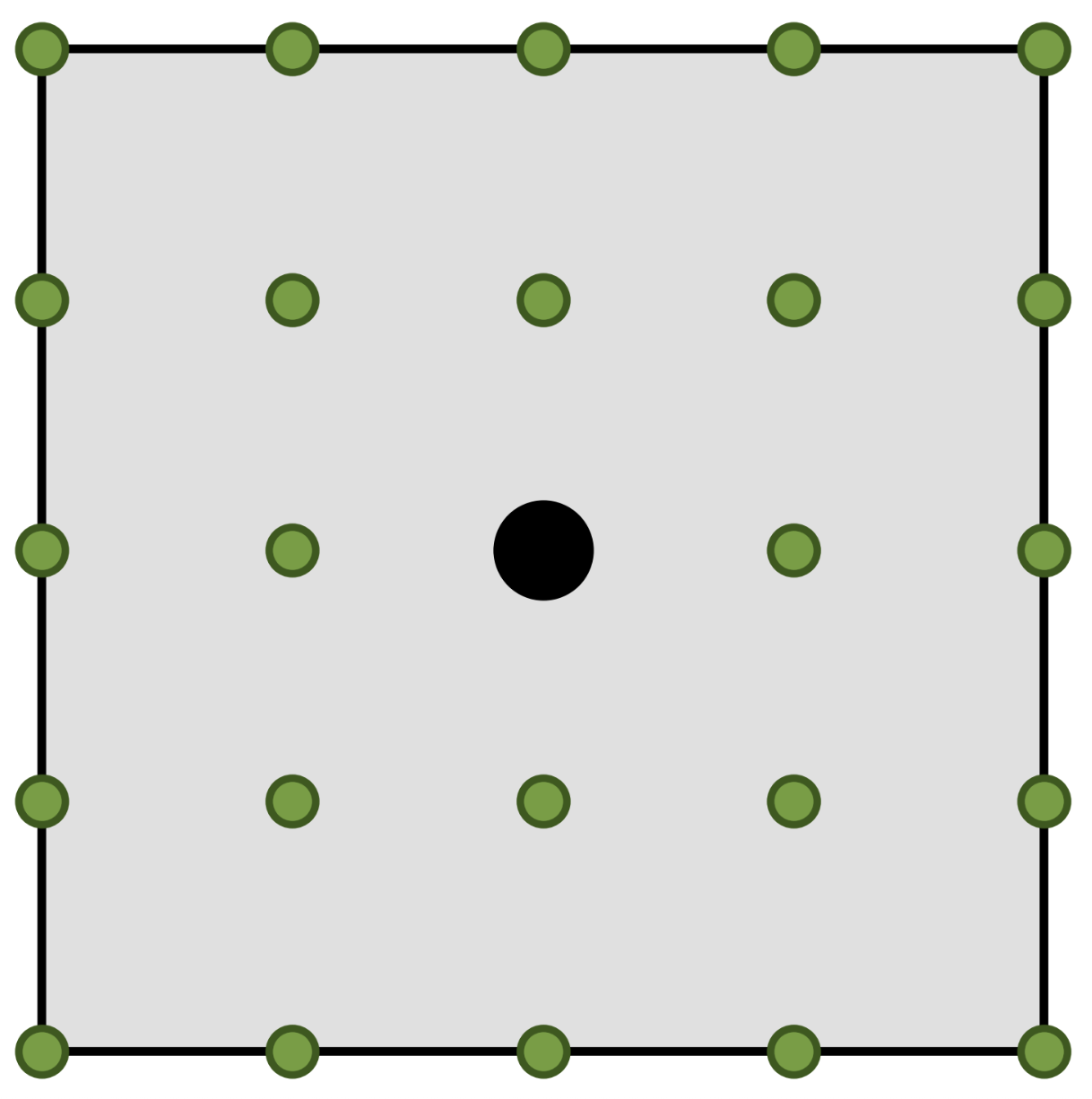}
        \caption{2D stencil for $m+1$.}
        \label{fig:stencil_2d_3}
    \end{subfigure}

    \vspace{1em} %

    \begin{subfigure}{0.3\textwidth}
        \centering
        \includegraphics[width=0.6\linewidth]{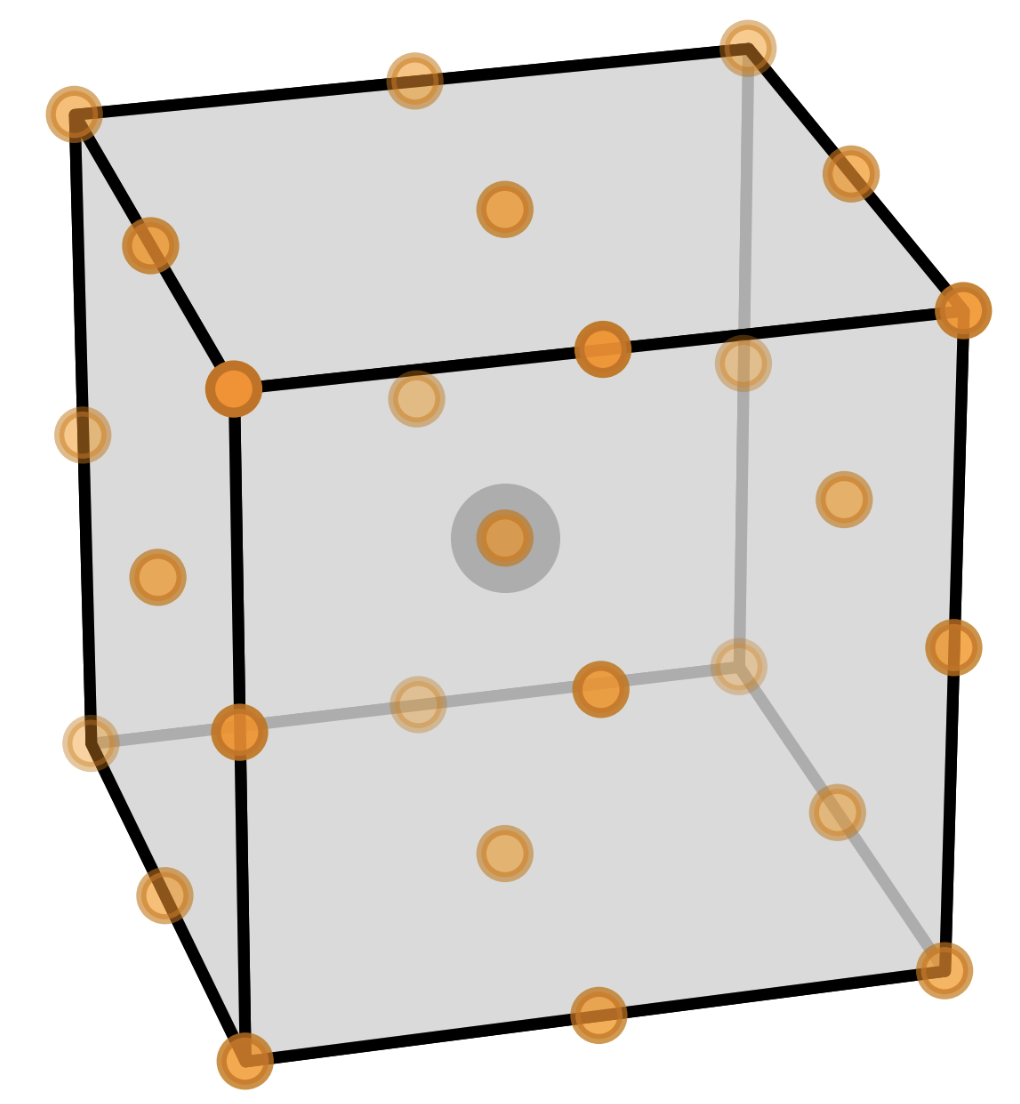}
        \caption{3D stencil for $m$.}
        \label{fig:stencil_3d_2}
    \end{subfigure}
    \begin{subfigure}{0.3\textwidth}
        \centering
        \includegraphics[width=0.6\linewidth]{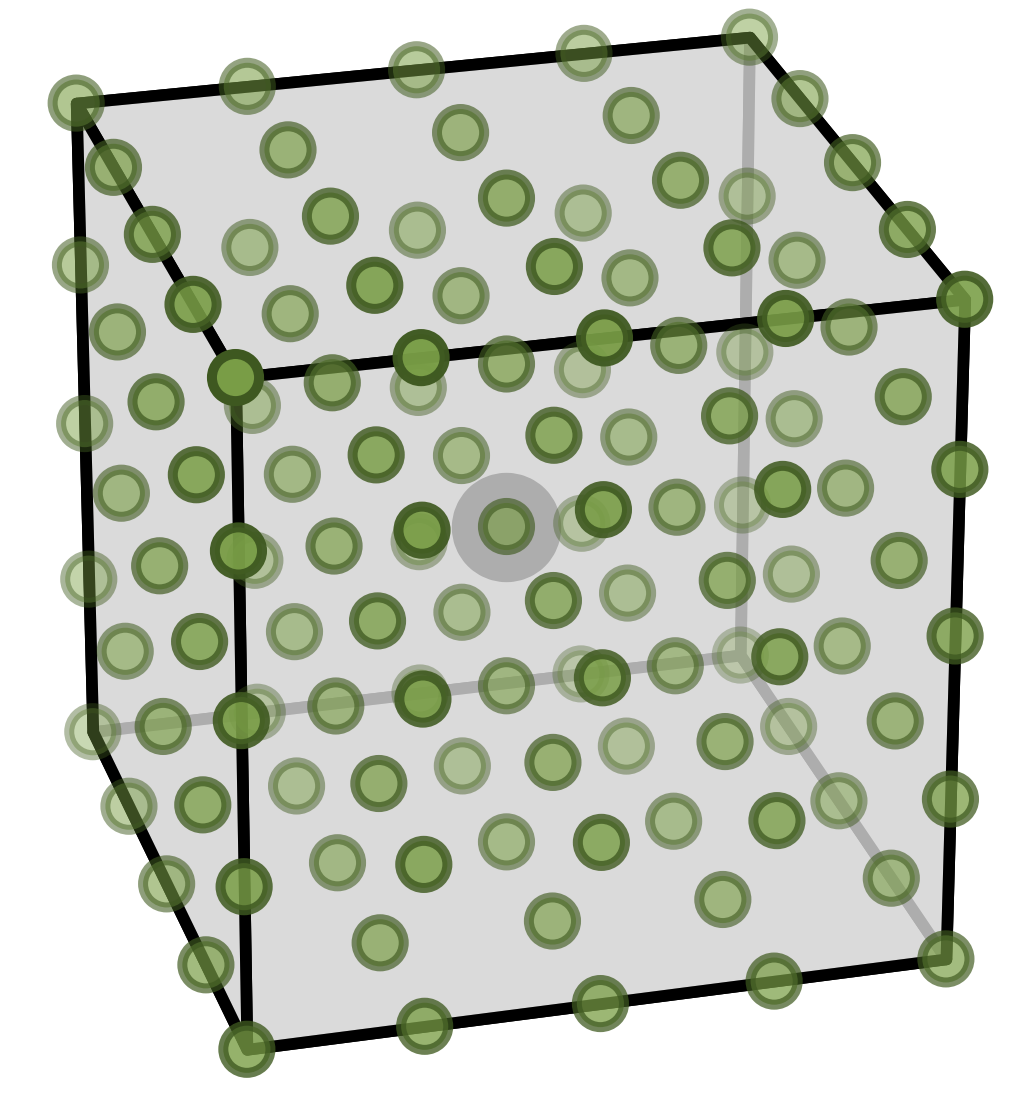}
        \caption{3D stencil for $m+1$.}
        \label{fig:stencil_3d_3}
    \end{subfigure}

    \caption{
        Visualization of local stencils in (a) one-dimensional, 
        (b--c) two-dimensional, and (d--e) three-dimensional settings.
    }
    \label{fig:local_stencils}
\end{figure}

\subsubsection{Discrete function values} \label{sec:uniform_subsampling}

When the target function is not available in closed form and only its values at a discrete set of data sites $X\subset\Omega\subset\mathbb{R}^d$ are given, we construct a nested family of local sample sets using the \emph{Uniform Subsampling Algorithm} described in \cite[Algorithm 1]{griebel2025kernel}, which we recall in the following. 
Let \(\boldsymbol{x}_c\in X\) be a chosen center and let \(\operatorname{knn}(\boldsymbol{x}_c)\) denote its \(k\)-nearest neighbours measured in the Chebyshev norm \(\|\cdot\|_\infty\). After an affine rescaling of these points into the unit cube \([0,1]^d\), we build a multilevel hierarchy by successive top–down uniform subsampling. For each scale \(m=0,1,\dots,M\) we partition \([0,1]^d\) into \(2^{md}\) axis-aligned cuboids of side length \(2^{-m}\). In every cuboid that contains at least one point of \(\operatorname{knn}(\boldsymbol{x}_c)\) we choose as representative the point nearest the cuboid midpoint; this representative is added to the level set \(X_m\) unless it was already selected at a coarser level.
This procedure yields a nested sequence of point sets
\(
X_1 \subset X_2 \subset \cdots \subset X_M,
\)
with fill distances satisfying
\[
h_{X_m,\Omega}\sim 2^{-m}.
\]
Selecting the point closest to the midpoint, rather than an arbitrary point in the cuboid, improves the separation radius. A visualization of the subsampling procedure for $m = 0,1,2,3$ starting from a set of 300 Halton points on the unit square is given in Figure \ref{fig:uniform_subsampling}.

\begin{figure}
    \centering
    \includegraphics[width=1\linewidth]{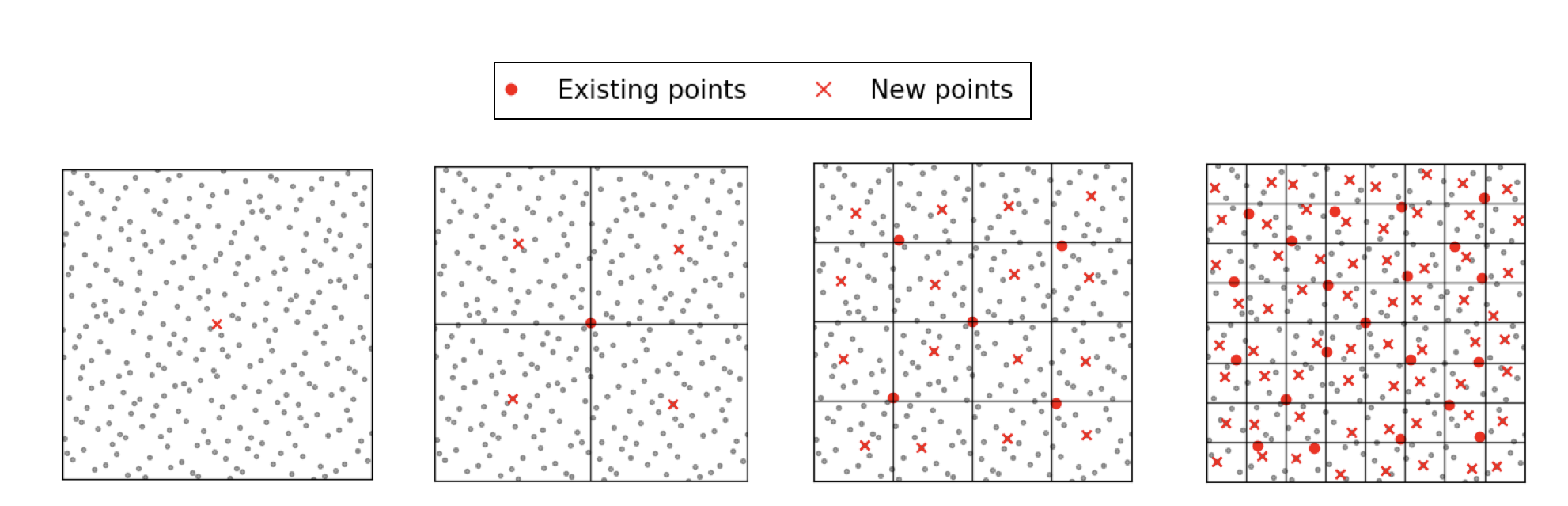}
    \caption{Uniform subsampling selection starting from 300 Halton points in the unit square $[0,1]^2$.}
    \label{fig:uniform_subsampling}
\end{figure}

\section{Numerical results}
\label{sec:num_results}
In this section, we apply the proposed method to local regularity detection in
one‑, two‑, and three‑dimensional signals.
The experiments have been performed on a MacBook Pro with an Apple M2 Max
processor and 32~GB of main memory. 

\subsection{One-dimensional setting.} 
We consider a piecewise function defined on $\Omega = [-1,1]$ as follows: 
\[
f(x)=
\begin{cases}
6, & x < -0.4,\\[6pt]
0.1\bigl| -20x - 9\bigr| + 6, & -0.4 \le x < -0.35,\\[6pt]
0.1\bigl| -20x - 5\bigr| + 6, & -0.35 \le x < -0.15,\\[6pt]
0.1\bigl| -20x - 1\bigr| + 6, & -0.15 \le x < -0.05,\\[6pt]
6 + \sin(20\pi x), & -0.05 \le x < 0.55,\\[6pt]
0.2\sin(6\pi x), & x \ge 0.55.
\end{cases}
\]
This function exhibits corner singularities at $x \in \{-0.35, -0.25, -0.15,-0.05\}$ 
and jump discontinuities at $x=\{-0.4, 0.55\}$.
Interpolation is performed using an $H^\tau(\mathbb{R})$ reproducing kernel with $\tau=3$. We first estimate the local smoothness exponent 
$\beta$ using the Fixed Stencils method introduced in Section~\ref{sec:stecntils_section}. We use $2\,000$ data sites, perform 8 successive refinements around each center, and set the kernel lengthscale parameter to twice the radius of the initial local stencil. The results are shown in the left column of Figure~\ref{fig:1d_jump_corners}.

We then apply the more general Uniform Subsampling method introduced in Section~\ref{sec:uniform_subsampling}, starting from $20\,000$ Halton points and selecting the $200$ nearest neighbours for each center. We again perform 8 refinements, and choose the lengthscale parameter to be twice the diameter of the interval containing the neighbours. The outcomes are displayed in the right column of Figure~\ref{fig:1d_jump_corners}.

At the jump discontinuities, the native-space estimate gives 
\(
\tilde{\gamma}_{\mathcal{N}} \approx -\tfrac{5}{2},
\)
which implies 
\(
\beta = \tau + \tilde{\gamma}_{\mathcal{N}} \approx \tfrac{1}{2}.
\)
The \(L_2\)-based estimator similarly yields 
\(
\tilde{\beta}_2 \approx \tfrac{1}{2}.
\)
Since a jump discontinuity belongs to \(H^{s}(O)\) for every \(s < \tfrac{1}{2}\), both approaches correctly identify the limiting smoothness \( \tfrac{1}{2} - \varepsilon\).
At the corner singularities, both estimates give 
\(
\beta \approx \tfrac{3}{2}.
\)
In this case, the true local regularity satisfies 
\( f \in H^{s}(O) \) for all \( s < \tfrac{3}{2} \), 
and therefore the method again detects the limiting value \( \tfrac{3}{2} - \varepsilon \) in agreement with theory.

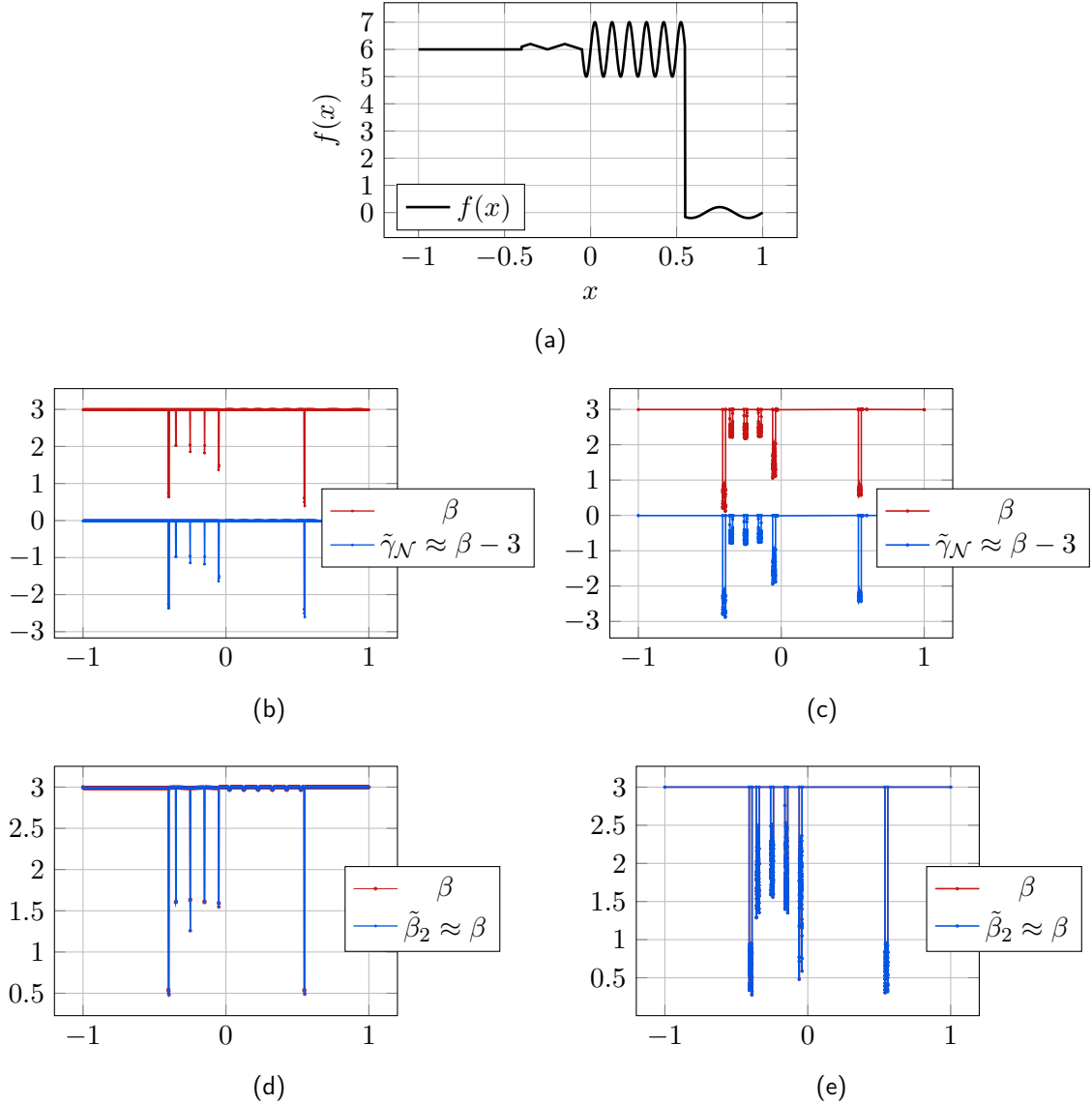
\begin{figure}[] %
  \centering
  \begin{subfigure}{0.8\textwidth}
    \centering
    \begin{tikzpicture}
      \begin{axis}[
        xlabel={$x$},
        ylabel={$f(x)$},
        grid=both,
        ytick distance=1,
        width=0.6\textwidth,
        height=0.4\textwidth,
        legend pos=south west
      ]
        \addplot[
          color=black,
          mark=none,
          line width=1pt
        ] table [x index=0, y index=1, col sep=space] {figure/1d_function.txt};
        \addlegendentry{$f(x)$}
      \end{axis}
    \end{tikzpicture}
    \caption{}
    \label{fig:1a}            
  \end{subfigure}
  
  \vspace{1em}

  \begin{subfigure}{0.49\textwidth}
    \begin{tikzpicture}
      \begin{axis}[
        grid=both,
        ytick distance=1,
        width=0.85\textwidth,
        height=0.68\textwidth,
         legend style={at={(1.4,0.6)}}
      ]
		\addplot[
          color={rgb,255:red,200; green,20; blue,20},
          mark=asterisk,
          mark size=0.5pt,
          line width=0.5pt
        ] 
        table [x index=0, y index=1, col sep=space] {figure/1d_tau_native_true.txt};
        \addlegendentry{$\beta$}
        \addplot[
          color={rgb,255:red,0; green,87; blue,231},
          mark=asterisk,
          mark size=0.5pt,
          line width=0.5pt
        ] 
        table [x index=0, y index=1, col sep=space] {figure/1d_tau_native.txt};
    \addlegendentry{$\tilde{\gamma}_\mathcal{N} \approx \beta-3$}
      \end{axis}
    \end{tikzpicture}
    \caption{}
    \label{fig:1b}            
  \end{subfigure}
\hfill
  \begin{subfigure}{0.49\textwidth}
    \centering
    \begin{tikzpicture}
        \begin{axis}[
        grid=both,
        ytick distance=1,
        width=0.85\textwidth,
        height=0.68\textwidth,
         legend style={at={(1.4,0.6)}}
      ]
        \addplot[
          color={rgb,255:red,200; green,20; blue,20},
          mark=asterisk,
          mark size=0.6pt,
          line width=0.6pt
        ] 
        table [x index=0, y index=1, col sep=space] {figure/1d_tau_native_true_halton.txt};
        \addlegendentry{$\beta$}
        \addplot[
          color={rgb,255:red,0; green,87; blue,231},
          mark=asterisk,
          mark size=0.6pt,
          line width=0.6pt
        ] 
        table [x index=0, y index=1, col sep=space] {figure/1d_tau_native_halton.txt};
    \addlegendentry{$\tilde{\gamma}_{\mathcal{N}} \approx \beta-3$}
      \end{axis}
    \end{tikzpicture}
    \caption{}
    \label{fig:1c}            
  \end{subfigure}
  \vspace{0.5 cm}
  
  \begin{subfigure}{0.49\textwidth}
    \begin{tikzpicture}
      \begin{axis}[
        grid=both,
        ytick distance=0.5,
        width=0.85\textwidth,
        height=0.68\textwidth,
         legend style={at={(1.3,0.6)}}
      ]
        \addplot[
          color={rgb,255:red,200; green,20; blue,20},
          mark=asterisk,
          mark size=0.8pt,
          line width=0.2pt
        ] 
        table [x index=0, y index=1, col sep=space] {figure/1d_tau_2.txt};
        \addlegendentry{$\beta$}
        \addplot[
          color={rgb,255:red,0; green,87; blue,231},
          mark=asterisk,
          mark size=0.5pt,
          line width=0.5pt
        ] 
        table [x index=0, y index=1, col sep=space] {figure/1d_tau_2.txt};
    \addlegendentry{$\tilde{\beta}_2\approx \beta$}
      \end{axis}
    \end{tikzpicture}
    \caption{}
    \label{fig:1d}            
  \end{subfigure}
  \hfill
  \begin{subfigure}{0.49\textwidth}
    \centering
    \begin{tikzpicture}
        \begin{axis}[
        grid=both,
        ytick distance=0.5,
        width=0.85\textwidth,
        height=0.68\textwidth,
         legend style={at={(1.3,0.6)}}
      ]
        \addplot[
          color={rgb,255:red,200; green,20; blue,20},
          mark=asterisk,
          mark size=0.6pt,
          line width=0.6pt
        ] 
        table [x index=0, y index=1, col sep=space] {figure/1d_tau_2_halton.txt};
        \addlegendentry{$\beta$}
        \addplot[
          color={rgb,255:red,0; green,87; blue,231},
          mark=asterisk,
          mark size=0.6pt,
          line width=0.6pt
        ] 
        table [x index=0, y index=1, col sep=space] {figure/1d_tau_2_halton.txt};
    \addlegendentry{$\tilde{\beta}_{2}\approx \beta$}
      \end{axis}
    \end{tikzpicture}
    \caption{}
    \label{fig:1e}            
  \end{subfigure}

  \caption{We analyze a univariate function shown in subfigure~\subref{fig:1a}, using a kernel of smoothness $\tau=3$. Subfigures \subref{fig:1b} and \subref{fig:1d} %
  display the results of singularity detection using the Fixed Stencils Method, while Subfigures \subref{fig:1c} and \subref{fig:1e} %
  show the results obtained with Halton points and the Uniform Subsampling method.}
  \label{fig:1d_jump_corners}
\end{figure}

\subsection{Two-dimensional setting.} 
We study the function shown in Figure~\ref{fig:composite_function}, 
which is the superposition of seven distinct components: a global sinusoidal modulation with frequency $4\pi$ in both directions;
a localized linear ramp confined to a rectangular subdomain;
 a ridge following a sinusoidal curve with linear decay perpendicular to the path;
a truncated cone centered at $(1.5, 4.5)$;
an exponential kernel centered at $(4.5, 2)$;
a damped oscillatory function
and a piecewise constant block over a rectangular support region.

Interpolation is performed using an $H^{\frac{7}{2}}(\mathbb{R}^2)$ reproducing kernel, i.e. $\tau = \frac{7}{2}$.
We apply both the Fixed Stencil Method, initialized with a grid of $100 \times 100$ points, and the Uniform Subsampling Algorithm, starting from $200\,000$ Halton points with 3 levels of refinement. Each local refinement is constructed from the $300$ nearest neighbors, and the lengthscale parameter is chosen as 
$h_{X,\Omega} \, \text{diam}(B)$, where $h_{X,\Omega}$
is the global fill distance and $\text{diam}(B)$
 is the diameter of the bounding box containing the neighboring points. The results are presented in Figure~\ref{fig:composite}.  

To illustrate the behavior of our quantities of interest, namely  
the native-space norm of each interpolant, and the squared discrete $L_2$-norm of the difference between successive interpolants,
Figure~\ref{fig:slope_composite} plots them against the fill distance for this experiment in a log--log scale. We emphasize that these slopes do not represent the final smoothness class, but the quantities denoted in the Algorithm \ref{alg:local_smoothness_detection} by $\tilde{\gamma}_\mathcal{N}$ and $\tilde{\beta}_2$,
from which the final smoothness $\beta$ can be infered directly.

\begin{figure}[]
    \centering
    \begin{subfigure}{0.32\textwidth}
        \centering
        \includegraphics[width=0.88\linewidth]{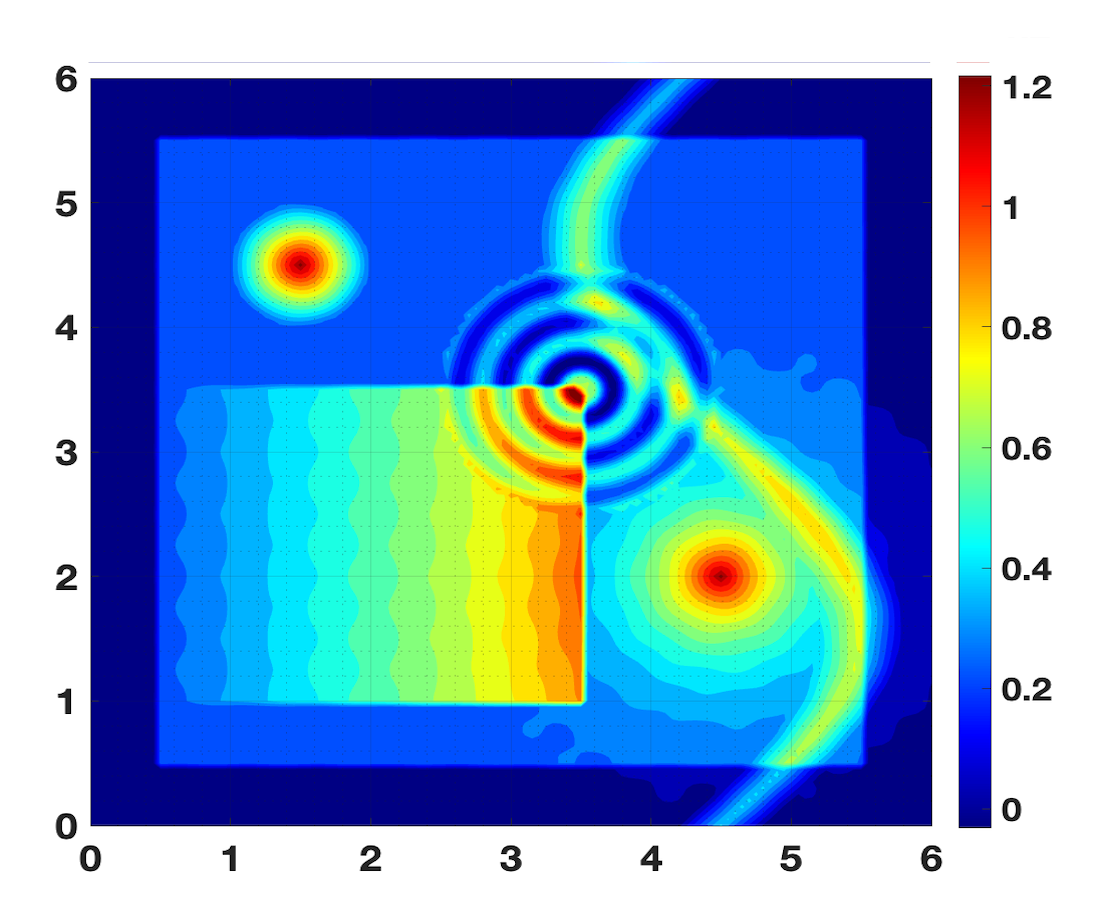}
        \caption{}
        \label{fig:composite_function}
    \end{subfigure}\hfill
    \begin{subfigure}{0.32\textwidth}
        \centering
        \includegraphics[width=0.95\linewidth]{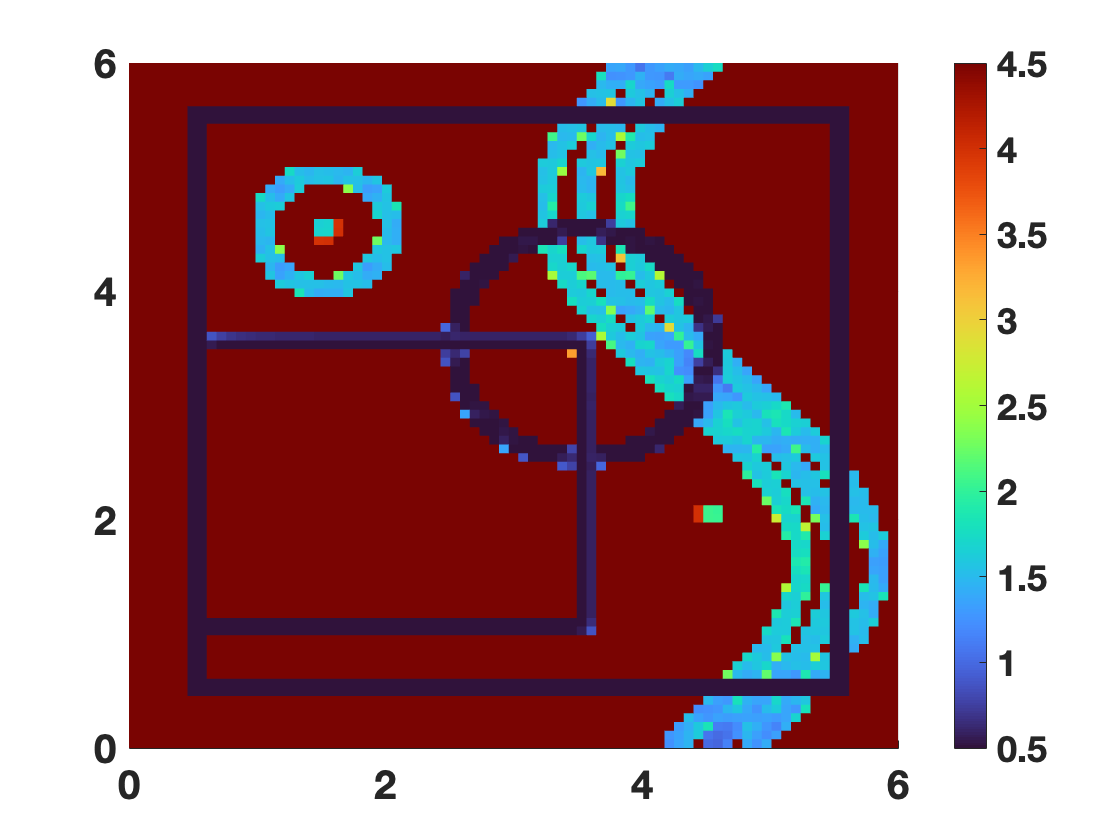}
        \caption{}
        \label{fig:composite_function_uniform_native}
    \end{subfigure}\hfill
    \begin{subfigure}{0.32\textwidth}
        \centering
        \includegraphics[width=0.95\linewidth]{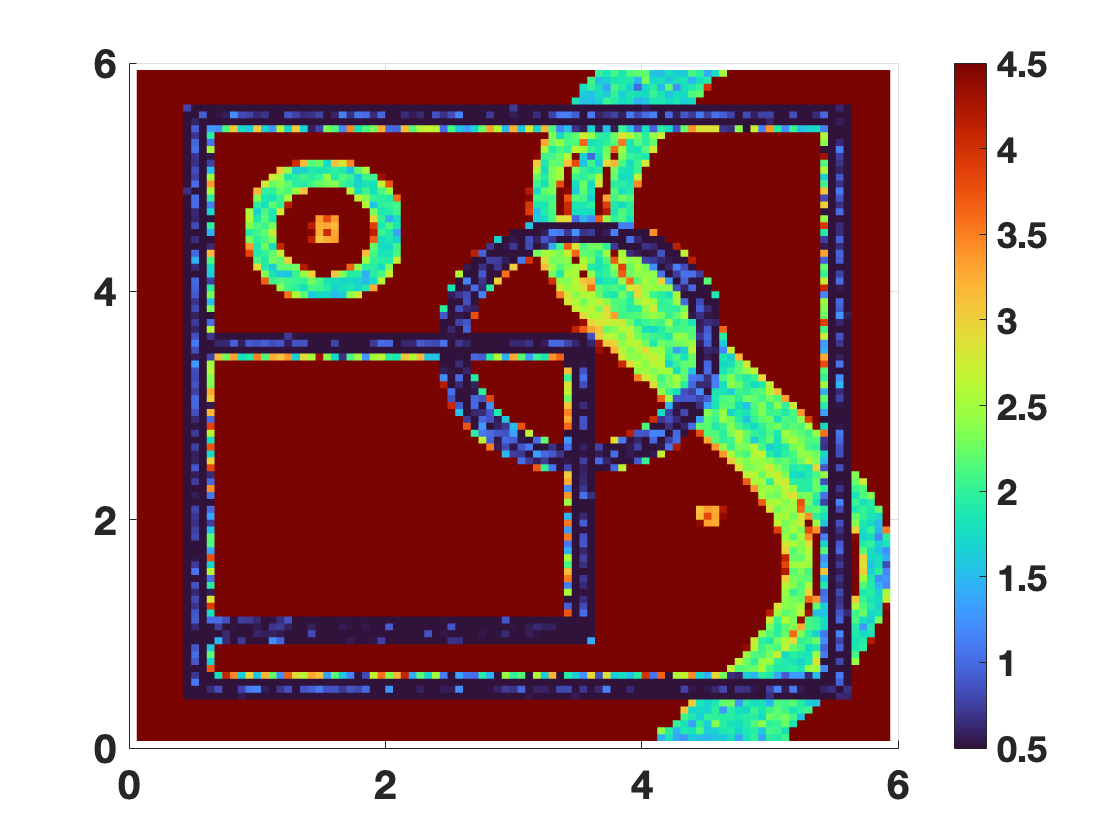}
        \caption{}
        \label{fig:composite_function_native}
    \end{subfigure}

    \vspace{2ex}

    \begin{subfigure}{0.32\textwidth}
        \centering
        \includegraphics[width=0.95\linewidth]{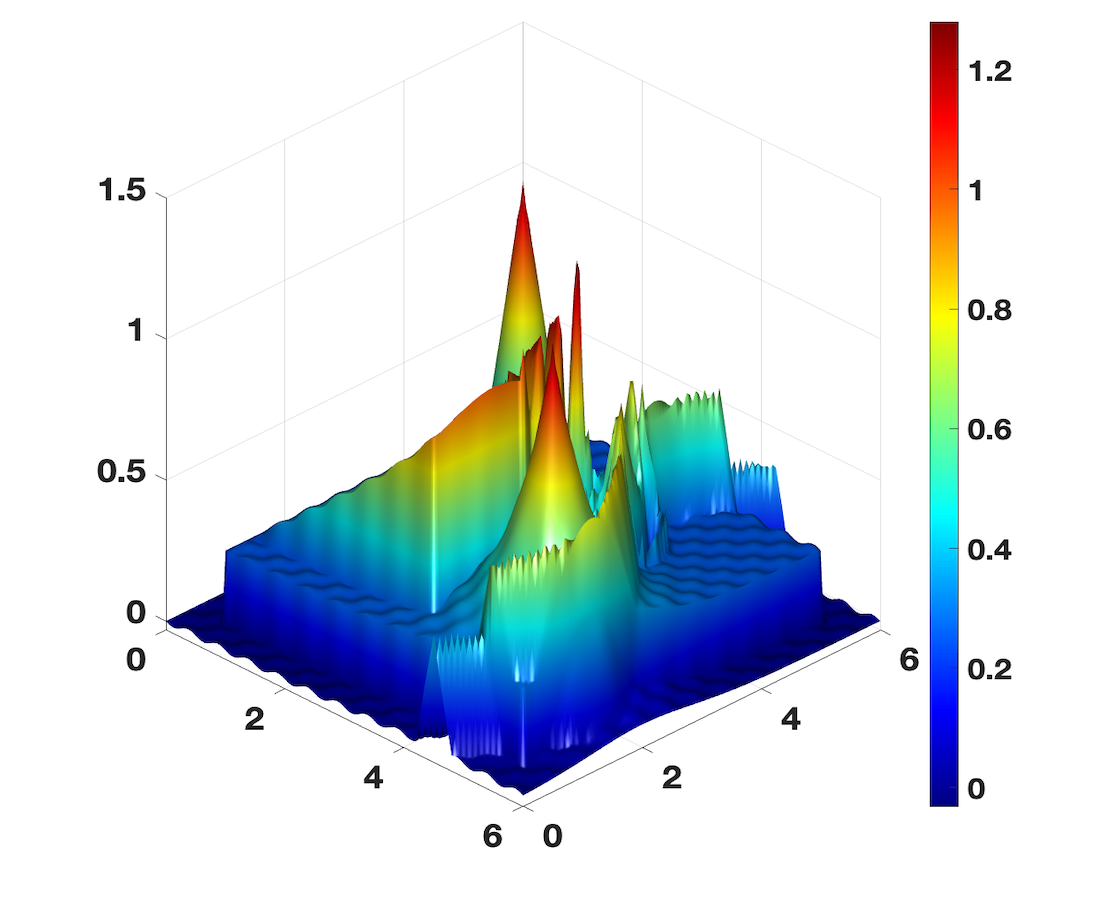}
        \caption{}
        \label{fig:composite_function2}
    \end{subfigure}\hfill
    \begin{subfigure}{0.32\textwidth}
        \centering
        \includegraphics[width=0.95\linewidth]{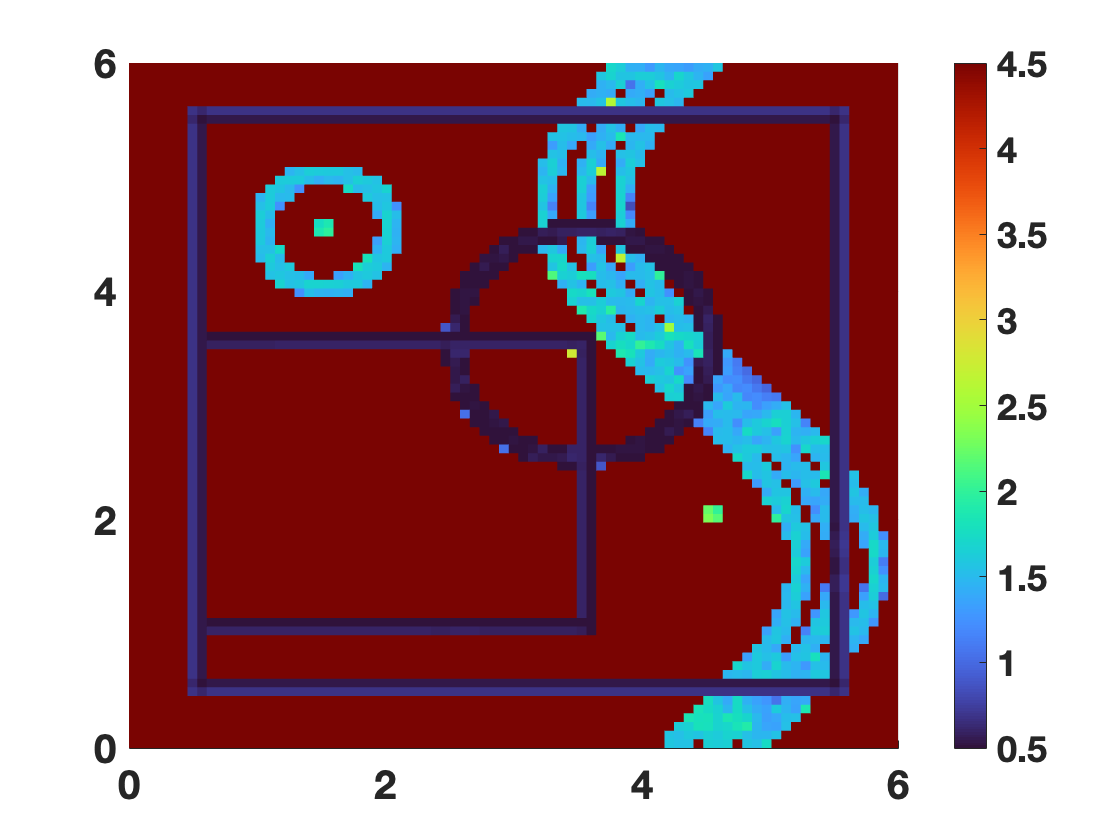}
        \caption{}
        \label{fig:composite_function_uniform_2}
    \end{subfigure}\hfill
    \begin{subfigure}{0.32\textwidth}
        \centering
        \includegraphics[width=0.95\linewidth]{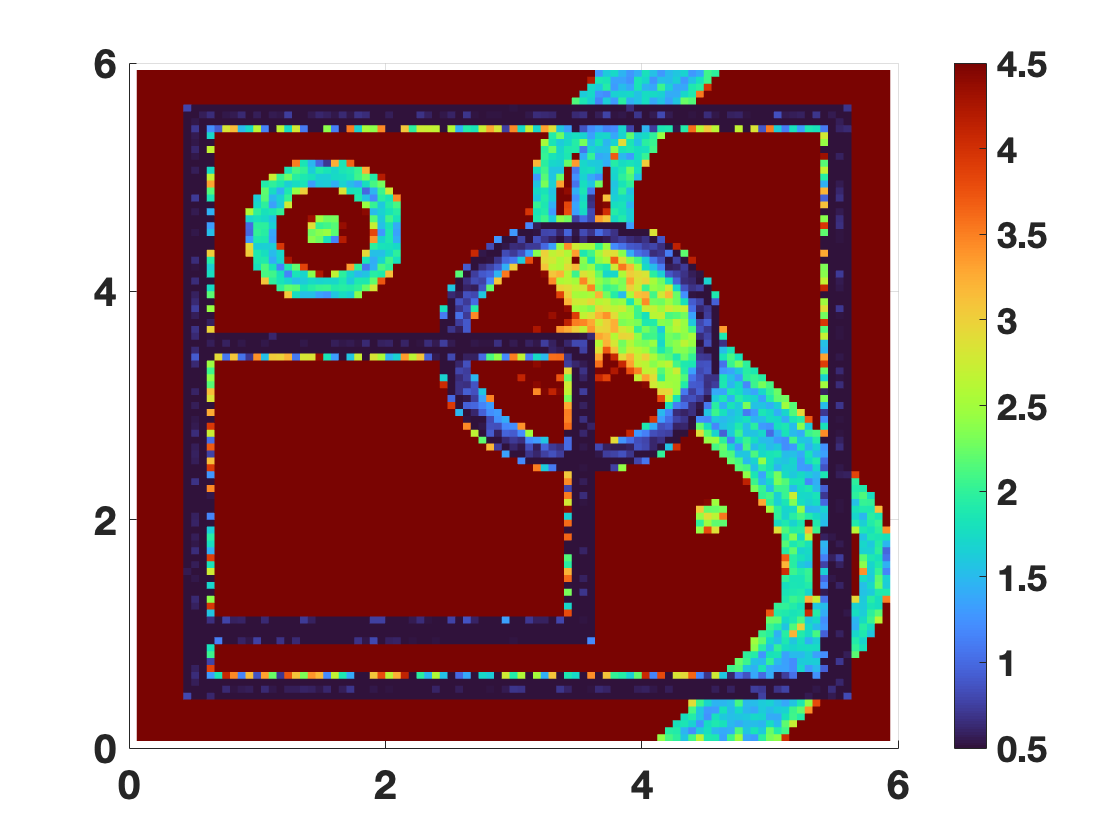}
        \caption{}
        \label{fig:composite_function_2}
    \end{subfigure}

    \caption{Subfigure~\subref{fig:composite_function} shows the 2D view of the test function and Subfigure~\subref{fig:composite_function2} its 3D view. Subfigures~\subref{fig:composite_function_uniform_native} and~\subref{fig:composite_function_uniform_2} display the native-space and $L_2$-norm results obtained with the Fixed Stencil Method on an initial $100\times100$ grid, respectively. Subfigures~\subref{fig:composite_function_native} and~\subref{fig:composite_function_2} show the corresponding results for the Uniform Subsampling approach applied to $200\,000$ Halton points.}
    \label{fig:composite}
\end{figure}
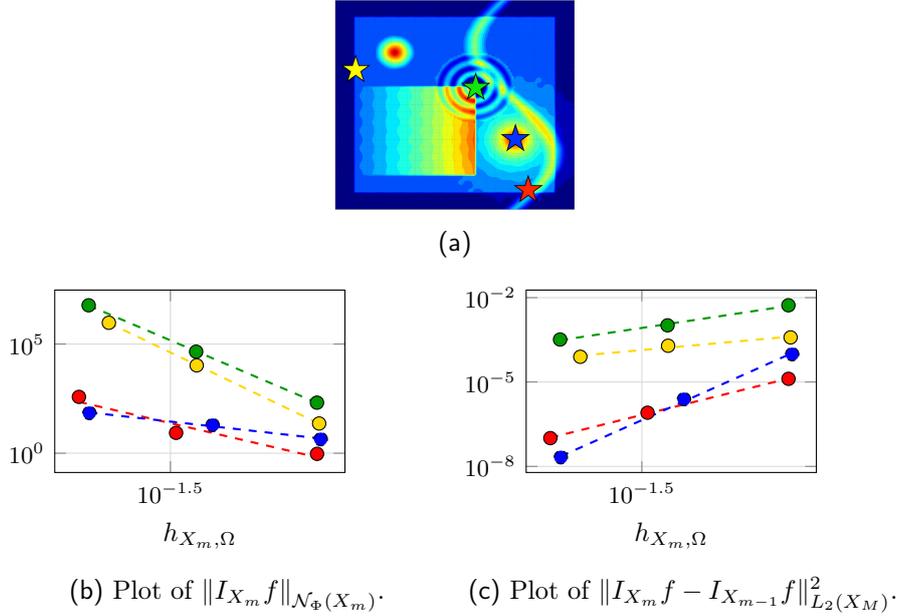
\begin{figure}
  \centering
  \begin{subfigure}{0.6\textwidth}
    \centering
    \includegraphics[width=0.35\linewidth]{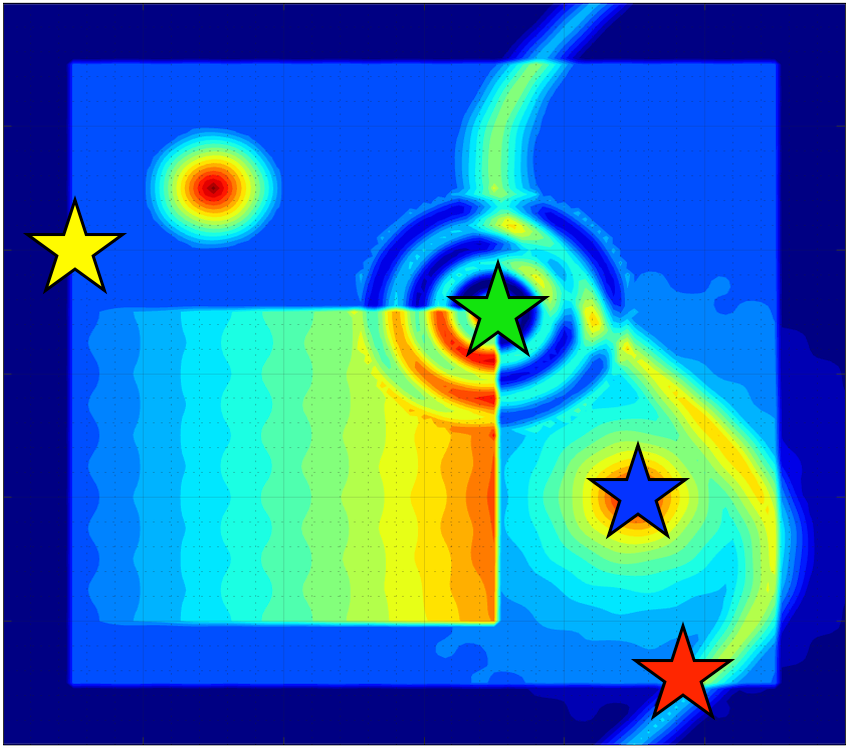}
    \caption{}
    \label{fig:stars}
  \end{subfigure}
  
  \vspace{0.5em}

  \begin{subfigure}[t]{0.4\textwidth}
    \begin{tikzpicture}
      \begin{axis}[width=0.9\linewidth, height=4cm, xmode=log, ymode=log, xlabel={$h_{X_m,\Omega}$}, grid=both, grid style={line width=0.4pt, draw=gray!30}, legend cell align=left, legend pos = north east, tick label style={font=\footnotesize}, label style={font=\small}]
      \addplot+[only marks, mark=*, mark options={draw=black, fill=gold}, mark size=2.5pt] table {figure/native_space_point_140770.dat};
      \addplot+[no marks, color=gold, dashed, thick] table {figure/fit_native_point_140770.dat};
      \addplot+[only marks, mark=*, mark options={draw=black, fill=red}, mark size=2.5pt] table {figure/native_space_point_321.dat};
      \addplot+[no marks, color=red, dashed, thick] table {figure/fit_native_point_321.dat};
      \addplot+[only marks, mark=*, mark options={draw=black, fill=darkgreen}, mark size=2.5pt] table {figure/native_space_point_995.dat};
      \addplot+[no marks, color=darkgreen, dashed, thick] table {figure/fit_native_point_995.dat};
      \addplot+[only marks, mark=*, mark options={draw=black, fill=blue}, mark size=2.5pt] table {figure/native_space_point_175287.dat};
      \addplot+[no marks, color=blue, dashed, thick] table {figure/fit_native_point_175287.dat};
      \end{axis}
    \end{tikzpicture}
    \caption{Plot of $\left\|I_{X_m} f\right\|_{\mathcal{N}_{\Phi}(X_m)}$.}
    \label{fig:stars_native}
  \end{subfigure}%
  \begin{subfigure}[t]{0.4\textwidth}
    \begin{tikzpicture}
      \begin{axis}[width=0.9\linewidth, height=4cm, 
      xmode=log, 
      ymode=log, 
      xlabel={$h_{X_m,\Omega}$}, 
      grid=both, 
      grid style={line width=0.4pt, draw=gray!30}, legend cell align=left, legend pos = north east, tick label style={font=\footnotesize}, label style={font=\small}]
      \addplot+[only marks, mark=*, mark options={draw=black, fill=gold}, mark size=2.5pt] table {figure/data_point_140770.dat};
      \addplot+[no marks, color=gold, dashed, thick] table {figure/fit_point_140770.dat};
      \addplot+[only marks, mark=*, mark options={draw=black, fill=red}, mark size=2.5pt] table {figure/data_point_321.dat};
      \addplot+[no marks, color=red, dashed, thick] table {figure/fit_point_321.dat};
      \addplot+[only marks, mark=*, mark options={draw=black, fill=darkgreen}, mark size=2.5pt] table {figure/data_point_995.dat};
      \addplot+[no marks, color=darkgreen, dashed, thick] table {figure/fit_point_995.dat};
      \addplot+[only marks, mark=*, mark options={draw=black, fill=blue}, mark size=2.5pt] table {figure/data_point_175287.dat};
      \addplot+[no marks, color=blue, dashed, thick] table {figure/fit_point_175287.dat};
      \end{axis}
    \end{tikzpicture}
    \caption{Plot of $\|  I_{X_{m}}f - I_{X_{m-1}}f \|_{L_2(X_M)}^2$.}
    \label{fig:stars_2}
  \end{subfigure}%
  \caption{Regression lines for four evaluation points with different smoothness classes, highlighted in Subfigure~\subref{fig:stars}. Figures~\subref{fig:stars_native} and~\subref{fig:stars_2} show the native space norm and $L_2$-norm of consecutive interpolant differences versus fill distance, with slopes of the fitted dashed lines corresponding to $\tilde{\gamma}_{\mathcal{N}} \approx \beta - \frac{7}{2}$ and $\tilde{\beta}_2 \approx \beta$, respectively.}
  \label{fig:slope_composite}
\end{figure}

Finally, to evaluate our method in a setting that more closely reflects practical applications, we consider a turbulence-like scenario derived from the numerical solution of the Navier--Stokes equations with temperature-dependent density. The computational setup and geometry are taken from the large fluid benchmark described in Section~2.13 of \cite{hecht2020freefem}. 
The dataset consists of approximately $100\,000$ sample points generated by \texttt{FREEFEM++} for this benchmark problem. We apply the Uniform Subsampling algorithm to this data using an $H^{\frac{7}{2}}(\mathbb{R}^2)$ reproducing kernel for the interpolation, with three refinement steps and $100$ nearest neighbours. 
The lengthscale parameter is selected using the same criterion adopted in the previous example. The resulting performance is illustrated in Figure~\ref{fig:turbolence_results}.

\begin{figure}[]
\vspace{0.5 cm}
    \centering
    \begin{subfigure}{0.32\textwidth}
        \centering
        \includegraphics[width=0.95\linewidth]{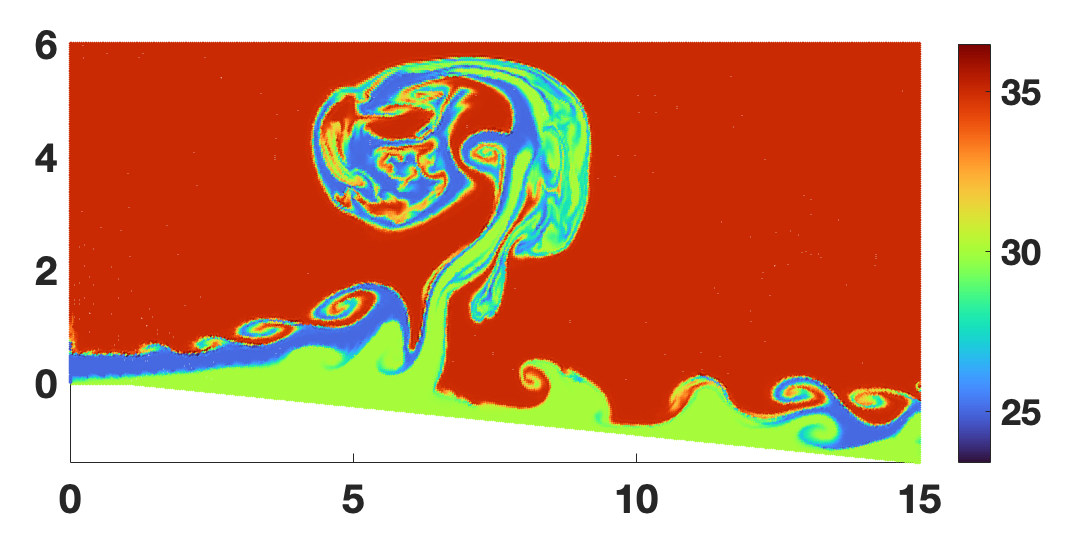}
        \caption{}
        \label{fig:turbolence_function}
    \end{subfigure}\hfill
    \begin{subfigure}{0.32\textwidth}
        \centering
        \includegraphics[width=0.95\linewidth]{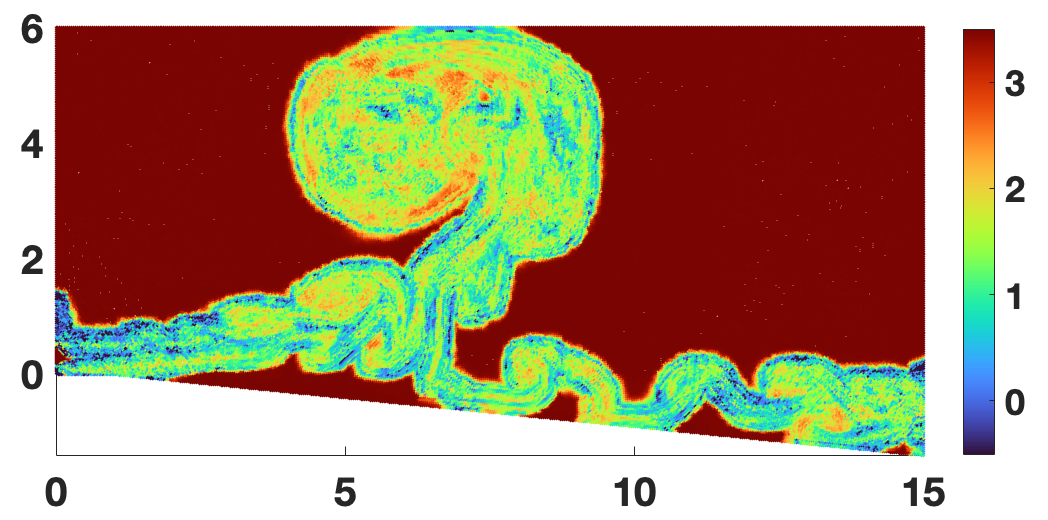}
        \caption{}
        \label{fig:turbolence_native}
    \end{subfigure}\hfill
    \begin{subfigure}{0.32\textwidth}
        \centering
        \includegraphics[width=0.95\linewidth]{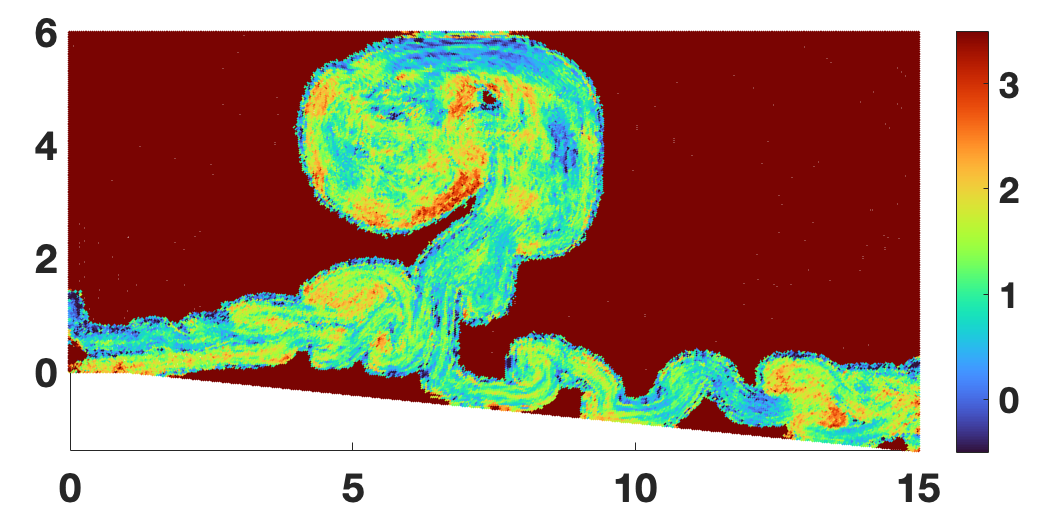}
        \caption{}
        \label{fig:turbolence_l2}
    \end{subfigure}

    \caption{Subfigure~\subref{fig:turbolence_function} displays the turbulence data, while Subfigures~\subref{fig:turbolence_native} and~\subref{fig:turbolence_l2} show the corresponding native-space and $L_2$-norm results obtained using the Uniform Subsampling method with approximately $100\,000$ points and $100$ neighbours.}
    \label{fig:turbolence_results}
\end{figure}

\subsection{Three-dimensional setting.} 
We analyze the following function on the complex volume of the \texttt{Stanford Bunny} from
the Stanford 3D Scanning 
Repository\footnote{\url{https://graphics.stanford.edu/data/3Dscanrep/}}:
\begin{equation}
f(x,y,z) =
\begin{cases}
\displaystyle \frac{4(z - c_z)}{\sqrt{(x - c_x)^2 + (y - c_y)^2 + (z - c_z)^2}} & \text{if } z > \frac{1}{2}\sin(5x + 2y), \\[1em]
1 & \text{otherwise},
\end{cases}
\end{equation}
where $\bs{c} = (c_x, c_y, c_z)$ is a center point located near the paw of the bunny. To avoid division by zero, the center point itself is excluded from the dataset.

Starting from the surface open-source mesh, a volumetric mesh of approximately $300\,000$ points was generated using \texttt{Gmsh}. Interpolation is performed using an $H^{3}(\mathbb{R}^3)$ reproducing kernel. The same computational configuration as in the preceding example is employed: the Uniform Subsampling Algorithm is applied with $300$ neighbors, and the lengthscale kernel parameter set to the global fill distance multiplied by the diameter of the bounding box containing the neighboring points. Finally, 5 refinement steps are performed. The resulting volume is shown in Figure~\ref{fig:bunny}.

\begin{figure}[]
    \begin{subfigure}{0.21\textwidth}
        \includegraphics[width=\linewidth]{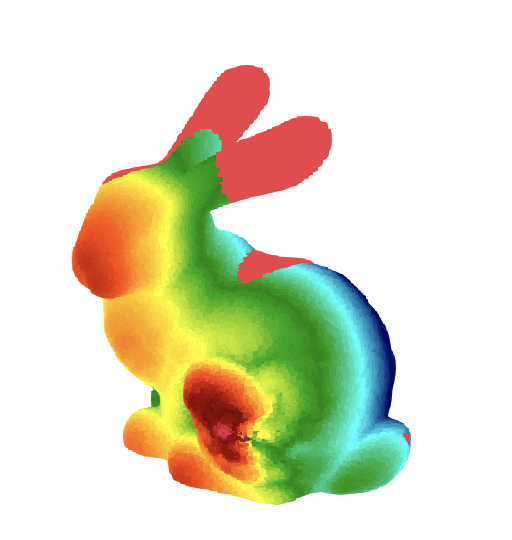}
        \caption{}
        \label{fig:bunny_function}
    \end{subfigure}
    \quad
    \begin{subfigure}{0.21\textwidth}
        \includegraphics[width=\linewidth]{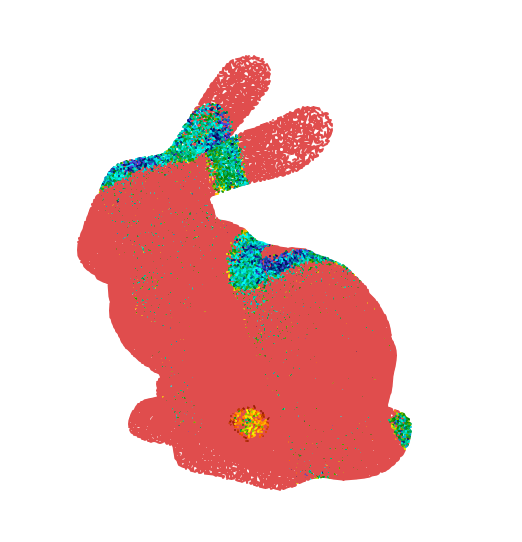}
        \caption{}
        \label{fig:bunny_native}
    \end{subfigure}
    \quad 
    \begin{subfigure}{0.21\textwidth}
        \includegraphics[width=\linewidth]{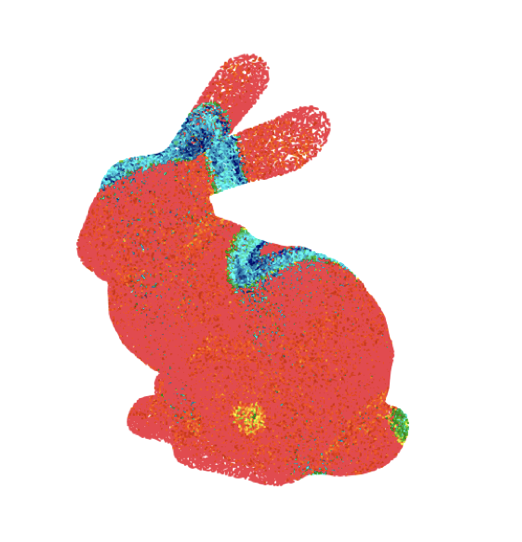}
        \caption{}
        \label{fig:bunny_2}
    \end{subfigure}
    \quad
    \begin{subfigure}{0.1\textwidth}
        \includegraphics[width=0.6\linewidth]{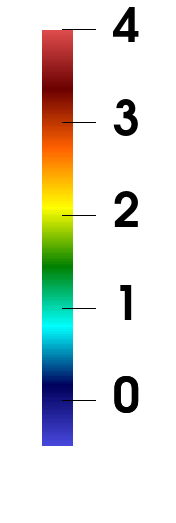}
        \caption*{}
    \end{subfigure}

    \caption{Subfigure~\subref{fig:bunny_function} shows the the test function on the \texttt{Stanford Bunny}, with a sinusoidal jump and a point singularity. Subfigures~\subref{fig:bunny_native} and~\subref{fig:bunny_2} display the native-space and $L_2$-norm results obtained with the Uniform Subsampling approach applied to $300\,000$ points.}
    \label{fig:bunny}
\end{figure}

\section{Conclusion and Outlook}
\label{sec:conclusion_outlook}
This work establishes a rigorous mathematical framework connecting function smoothness to approximation convergence rates through strengthened direct and inverse statements for finitely smooth kernels. Our theoretical contribution is threefold. First, we generalized classical direct error estimates from $L_2(\Omega)$ to arbitrary $L_q(\Omega)$ norms. Then, we strengthened existing results in inverse approximation theory by weakening one of the required hypotheses in previous theoretical results, i.e., we require the existence of a single sequence of quasi-uniformly distributed points in place of uniform decay across all quasi-uniform sets. 
Finally, we established novel direct and inverse statements, which are formulated in the native space norm $\| \cdot \|_{\mathcal{N}_\Phi(\Omega)}$ directly, complementing the standard $L_2(\Omega)$-based theory.

These theoretical statements establish a one-to-one correspondence between function smoothness and approximation rates, creating the foundation for \emph{SALSA} (Sobolev Algorithm for Local Smoothness Analysis). This algorithm detects locally not only the non-smooth regions, but also the smoothness class of the different singularities.
Finally, numerical experiments demonstrate SALSA's effectiveness across univariate and multivariate scattered data. We tested the algorithm on problems ranging from one to three spatial dimensions with datasets containing up to hundreds of thousands points, considering both smooth regions and various singularity types including jump discontinuities, corner singularities, and complex geometric features. In all cases, SALSA accurately identified the local smoothness classes, validating the theoretical predictions.

\bigskip\noindent
{\bf Acknowledgement.}
SA was funded by the SNSF starting grant 
``Multiresolution methods for unstructured data'' (TMSGI2\_211684). LL was supported by the General Research Fund (GRF Nos. 12300922, 12301824) of the Hong Kong Research Grants Council, and by the Guangdong and Hong Kong Universities ``1 + 1 + 1'' Joint Research Collaboration Scheme (Project No. 2025A0505000014).

\bibliographystyle{plain}
\bibliography{bibliography}

\newpage

\end{document}